\documentclass[a4paper,10pt]{amsart}
\usepackage[latin1]{inputenc}
\usepackage[english]{babel}
\usepackage{amssymb,amsmath,times,amsfonts}
\usepackage{graphicx}
\usepackage[usenames,dvipsnames]{color}
\usepackage{xcolor}
\usepackage{bbm}
\usepackage{pgf}
\def\epsilon{\varepsilon}

\def\natnums{\mathbb N}
\def\reals{\mathbb R}
\def\R{\reals}

\def\N{\natnums}

\newcommand\eps{\ensuremath{\varepsilon}}

\newcommand{\norm}[1]{\left\Vert#1\right\Vert}

\def\weak{{\omega}}

\newtheorem{theo}{Theorem}[section]
\newtheorem{lem}[theo]{Lemma}%[section]
\newtheorem{pro}[theo]{Proposition}%[section]
\newtheorem{cor}[theo]{Corollary}%[theo]

\theoremstyle{definition}
\newtheorem{defi}[theo]{Definition}%[section]

\theoremstyle{remark}
\newtheorem{rem}[theo]{Remark}%[theo]

\numberwithin{equation}{section}

\newcommand{\supp}{\operatorname{supp}}

\newcommand{\ee}{\varepsilon}

\newcommand{\la}{\langle}
\newcommand{\ra}{\rangle}

\newcommand{\dualC}[1][K]{\mathcal{M}(#1)}
\def\dmu{\;{\rm d}\mu}

\newcommand{\cantor}{\mathcal{C}}
\newcommand{\tcantor}{2^{<\mathbb{N}}}
\newcommand{\mm}[2]{\tilde{m}^{\scriptscriptstyle (#1)}_{\scriptscriptstyle #2}}
\newcommand{\mf}[1]{m_{\scriptscriptstyle #1}}

\title{On the Bishop-Phelps-Bollob\'as property for numerical radius in $C(K)$ spaces}

\author{A. Avil\'{e}s}
\address{Departamento de Matem\'{a}ticas\\
Facultad de Matem\'{a}ticas\\ Universidad de Murcia\\ 30100 Espinardo (Murcia)\\
Spain} \email{avileslo@um.es}

\author{A. J. Guirao}
\address{Instituto Universitario de Matem\'{a}tica Pura y Aplicada\\ 
Universidad Polit\'ecnica de Valencia\\ Camino de Vera s/n\\ 46022 Valencia\\ Spain}\email{anguisa2@mat.upv.es}

\author{J. Rodr\'{i}guez}
\address{Departamento de Matem\'{a}tica Aplicada\\
Facultad de Inform\'{a}tica\\ Universidad de Murcia\\ 30100 Espinardo (Murcia)\\
Spain} \email{joserr@um.es}

\date{\today}

\dedicatory{Dedicated to Irene}



\thanks{Research supported by {\em Ministerio de Econom\'{i}a y Competitividad} and FEDER 
under project MTM2011-25377. 
A. Avil\'{e}s was supported by {\em Ram\'{o}n y Cajal} contract (RYC-2008-02051). A. J. Guirao was supported by Generalitat Valenciana (GV/2010/036).}

\begin{document}

\begin{abstract}
We study the Bishop-Phelps-Bollob\'{a}s property for numerical radius within the framework of $C(K)$ spaces. 
We present several sufficient conditions on a compact space~$K$ ensuring 
that $C(K)$ has the Bishop-Phelps-Bollob\'{a}s property for numerical radius. 
In particular, we show that $C(K)$ has such property whenever $K$ is metrizable.
\end{abstract}

\maketitle

\section{Introduction}

The Bishop-Phelps-Bollob\'{a}s property for numerical radius has been recently introduced 
in~\cite{GK} as a quantitative way of studying the set of operators on a Banach space
that attain their numerical radius (see below for precise definitions). 
Since Sims \cite{Sims} raised the question of the norm denseness of the set of numerical radius 
attaining operators, several results have been obtained in this direction.  
Acosta initiated a systematic study of this problem in her Ph.D. Thesis~\cite{AcostaThesis}, 
followed by~\cite{ARenorming} and joint works with Pay\'{a}~\cite{APsecondadj, APRNP}.  
Prior to them, Berg and Sims~\cite{BergSims} gave a positive answer for uniformly convex spaces 
and Cardassi obtained positive answers for $\ell_1$, $c_0$, $C(K)$ ($K$ compact metric space), 
$L_1(\mu)$ and uniformly smooth spaces, see~\cite{Card-USm,Card-c0,Card-CK}.
Note that Johnson and Wolfe~\cite{JW} had already shown that the set of norm attaining operators $T\colon C(K)\to C(L)$ 
is norm dense in the space of operators $\mathfrak{L}(C(K), C(L))$, where 
$K$ and $L$ are arbitrary compact spaces.  Acosta~\cite{AcostaThesis} 
pointed out that an operator $T\colon C(K)\to C(K)$ attains its norm if and only if it 
attains it numerical radius. This observation together with Johnson and Wolfe's result 
led her to conclude that the set of numerical radius attaining operators on $C(K)$ 
is dense in~$\mathfrak{L}(C(K))$.

Using a renorming of~$c_0$, Pay\'{a}~\cite{Paya} provided an example of a Banach space 
$X$ such that the set of numerical radius attaining operators on~$X$ is not norm dense
in~$\mathfrak{L}(X)$, answering in the negative Sims' question. Acosta, 
Aguirre and Pay\'{a}~\cite{AAP} gave another counterexample: 
$X=\ell_2 \oplus_{\infty}G$, where $G$ is Gowers' space. 
Observe that these examples show that there exist Banach spaces failing the 
Bishop-Phelps-Bollob\'{a}s property for numerical radius.

In~\cite{GK} it is shown that $\ell_1$ and $c_0$ have the Bishop-Phelps-Bollob\'{a}s property 
for numerical radius. In fact, the proof for~$c_0$ can be reduced to a duality argument 
from the proof for~$\ell_1$. In this paper we focus on the Banach space $C(K)$ and we
discuss whether this space has the Bishop-Phelp-Bollob\'{a}s property for numerical radius. 
Trying to transfer the ideas in~\cite{GK} to the $C(K)$ case is clearly not enough. 

We now summarize briefly the contents of this paper.

In Section~\ref{section2} we introduce the concepts of \emph{compensation} of a regular measure
and of compact space admitting \emph{local compensation} (Definition~\ref{defi:compensation}).  
These notions are essential tools for our proofs and are applied
to obtain a parametric version of the classical Bishop-Phelps-Bollob\'{a}s theorem for functionals on~$C(K)$
(Lemma~\ref{lBPB-C(K)}). Then we show that $C(K)$ has the Bishop-Phelps-Bollob\'{a}s property for 
numerical radius whenever $K$ admits local compensation (Theorem~\ref{lBPBP-C(K)}). 

In Section~\ref{section3} we show that every compact metric space 
admits local compensation. In fact, a stronger result holds true, namely, that 
every compact metric space admits a \emph{compensation function} (Definition~\ref{defi:MCompensation}).  
We rely on the constructive proof that the Cantor set 
admits a compensation function (Theorem~\ref{theo:Cantor}) and the fact that 
compensation functions can be transferred to other compacta 
via regular averaging operators (Lemma~\ref{pro:QuotientCompensation}). 
As a consequence of Theorem~\ref{lBPBP-C(K)}, it turns out that
$C(K)$ has the Bishop-Phelps-Bollob\'{a}s property for numerical radius whenever $K$ is metrizable. 

In Section~\ref{section4} we discuss the case of non-metrizable compacta. 
With the help of the auxiliary concept of \emph{closeness function}, we present two 
examples of compact spaces admitting local compensation but no compensation function 
(Theorems~\ref{theo:AGamma} and \ref{theo:ordinalAA}). 
We also show that there exist compact spaces that do not admit local compensation.
We finish the paper with some open problems, see Subsection~\ref{section5}.

\subsection*{Terminology}

By countable we mean finite or countably infinite. The first uncountable ordinal
is denoted by~$\omega_1$.
All our Banach spaces~$X$ are real. We write 
$$
	B_{X}=\{x\in X\colon \|x\|\leq 1\}
	\quad\mbox{and} 
	\quad 
	S_X=\{x\in X\colon\|x\|=1\}.
$$ 
The topological dual of~$X$ is denoted by~$X^*$ and the weak$^*$ topology on~$X^*$ is
denoted by~$\weak^*$. The evaluation of $x^*\in X^*$ at $x\in X$
is denoted by $x^*(x)=\la x^*,x \ra = \la x, x^* \ra$. We write
$\Pi(X)=\{(x,x^*)\in S_X \times S_{X^*}\colon x^*(x)=1\}$. 
We write 
$$
	\pi_2(x)=\{x^*\in B_{X^*}\colon x^*(x)=1\} 
	\quad\mbox{and}
	\quad
	\pi_2(x,\delta)=\{x^*\in B_{X^*}\colon\,x^*(x)\geq 1-\delta\}
$$
for every $x\in B_X$ and $\delta>0$. 
By an operator on~$X$ we mean a linear continuous mapping $T\colon X \to X$. Its numerical radius
is defined by
$$
	\nu(T)=\sup\{ |\la x^*, T(x) \ra| \colon (x,x^*)\in \Pi(X)\}.
$$
The Banach space of all operators on~$X$ is denoted by~$\mathfrak{L}(X)$. 
It is well known that $\nu(\cdot)$ is a continuous seminorm on~$\mathfrak{L}(X)$.
In general, there exists a constant $n(X) \geq 0$ (the \emph{numerical index} of~$X$) such that
\[
	n(X)\norm{T}\leq\nu(T)\leq\norm{T} \quad\text{for all }T\in\mathfrak{L}(X).
\]
For background in numerical radius (resp. index) we refer to~\cite{BD1,BD2} 
(resp.~\cite{Kad-Mart-Paya}).
The Bishop-Phelps-Bollob\'{a}s property we are concerned about is defined as:

\begin{defi}\label{defi:BPBpNR}
We say that a Banach space $X$ has the {\em Bishop-Phelps-Bollob\'as (BPB) property for numerical radius}
if there is a function $\delta\colon (0,1)\to(0,1)$ such that: for every $0<\eps<1$, 
$T\in \mathfrak{L}(X)$ with $\nu(T)=1$ and $(x,x^*)\in \Pi(X)$ with $\la x^*,T(x) \ra \geq 1-\delta(\eps)$,
there exist $T_0\in \mathfrak{L}(X)$ with $\nu(T_0)=1$ and $(x_0,x_0^*)\in \Pi(X)$ with $\la x_0^*, T_0(x_0)\ra=1$ such that
$\nu(T-T_0)\leq \epsilon$, $\|x-x_0\|\leq \epsilon$ and $\|x^*-x_0^*\|\leq \epsilon$.
\end{defi}

Let $K$ be a compact space (i.e. compact Hausdorff topological space). We denote by $C(K)$ 
the Banach space of all continuous real-valued functions on~$K$ (equipped with the supremum norm).
It is known that $n(C(K))=1$ and therefore $\nu(T)=\|T\|$ for every $T\in \mathfrak{L}(C(K))$. 
Given any $f\in C(K)$ and $r\in \R$, we freely use notations like
$\{f\leq r\}=\{t\in K\colon f(t) \leq r\}$. The dual $C(K)^*$ is identified (via Riesz's theorem)
with the Banach space $\dualC$ of all regular Borel (signed) measures on~$K$ (equipped
with the total variation norm). We write $\mathcal{M}^+(K)=\{\mu \in \dualC\colon \mu \geq 0\}$. 
For every $t\in K$ we denote by $\delta_t\in \dualC$ the Dirac measure at~$t$. As usual,
given any $\mu\in \dualC$, we write $|\mu|$, $\mu^+$ and $\mu^-$ to denote, respectively, the variation, positive
part and negative part of~$\mu$. By a Hahn decomposition of~$\mu$ we mean a partition $(P,N)$
of~$K$ into Borel sets such that $\mu(B)\geq 0$ (resp. $\mu(B)\leq 0$) for every Borel
set $B \subseteq P$ (resp. $B \subseteq N$). The support of~$\mu$
is denoted by ${\rm supp}(\mu)$. Given $\mu_1,\mu_2 \in \dualC$, 
we write $\mu_1 \ll \mu_2$ (resp. $\mu_1 \perp \mu_2$)
if $\mu_1$ is absolutely continuous with respect to~$\mu_2$
(resp. $\mu_1$ and $\mu_2$ are mutually singular).

\section{BPB property for numerical radius in $C(K)$}\label{section2}

Throughout this section $K$ is a fixed compact space. Our aim 
is to give a sufficient condition ensuring that $C(K)$ has the BPB
property for numerical radius, namely, that $K$ admits {\em local compensation} (see the following definition).
In Sections~\ref{section3} and~\ref{section4} we shall prove that $K$ admits local compensation whenever it is metrizable, as well as
in other cases.

\begin{defi}\label{defi:compensation}
Let $W(K)$ be the set of all $\weak^*$-continuous functions $F\colon K\to B_{\dualC}$.
\begin{enumerate}

\item[(i)] We say that $\nu\in \dualC$ is a \emph{compensation} of $\mu\in\dualC$ provided that: 
\begin{itemize}
\item $0\leq \nu \leq \mu^+$ and $\nu(K)=\mu(K)$ if $\mu(K)> 0$;
\item $\nu=0$ if $\mu(K)\leq 0$.
\end{itemize}

\item[(ii)] We say that $G \in W(K)$ is a {\em compensation} of~$F\in W(K)$ 
if $G(t)$ is a compensation of~$F(t)$ for every $t\in K$.

\item[(iii)] We say that $K$ admits {\em local compensation} if every 
element of~$W(K)$ admits a compensation.
\end{enumerate}
\end{defi}

\begin{theo}\label{lBPBP-C(K)}
If $K$ admits local compensation, then $C(K)$ has the BPB 
property for numerical radius.
\end{theo}

In order to prove Theorem~\ref{lBPBP-C(K)} we need several lemmas.
Let us first point out that compensations of single measures always exist:

\begin{rem}\label{rem:ExistenceSingle}
If $\mu\in \dualC$ satisfies $\mu(K)>0$ and we set $\lambda:=\frac{\mu(K)}{\mu^+(K)}\in (0,1]$, then 
$\nu:=\lambda\mu^+$ is a compensation of~$\mu$.
\end{rem}

\begin{lem}\label{simplifynorm}
If $\nu\in \dualC$ is a compensation of $\mu\in\dualC$, then 
$\|\mu-\nu\|\leq 2\|\mu^-\|$ and $\|\nu\|\leq \|\mu\|$.
\end{lem}
\begin{proof} This is obvious if $\mu(K)\leq 0$. Suppose 
$\mu(K) > 0$. 
Since $(\mu^+-\nu) \perp \mu^-$, we have
\begin{multline*}
	\norm{\mu-\nu} = 
	\norm{(\mu^+ - \nu) - \mu^-} = 
	\norm{\mu^+ - \nu}+\norm{\mu^-} = \\ =
	(\mu^+ - \nu)(K) + \mu^-(K)=\mu^+(K)-\mu(K)+\mu^-(K)
	=2\mu^-(K)=2\norm{\mu^-}.
\end{multline*}
On the other hand, $\|\nu\|=\nu(K)=\mu(K) \leq \|\mu\|$. 
\end{proof}

\begin{lem}\label{lemmapi}
Let $(f,\mu)\in S_{C(K)}\times S_{\dualC}$ and let $(P,N)$ be a Hahn decomposition of~$\mu$. 
Then $\mu(f)=1$ if and only if 
$$
	|\mu|\big((\{f=1\}\cap P) \cup (\{f=-1\}\cap N)\big)=1.
$$
\end{lem}

\begin{proof}
Write $A:=(\{f=1\}\cap P) \cup (\{f=-1\}\cap N)$. Observe first that
\begin{equation}\label{setA}
	\int\limits_{A} f\,\dmu=
	\int\limits_{\{f=1\}\cap P} f\,\dmu + 
	\int\limits_{\{f=-1\}\cap N} f\,\dmu = 
	|\mu|(A).
\end{equation}
Therefore, if $|\mu|(A)=1$ then $\mu(f)=\int_{A} f\,\dmu=1$.

Conversely, if $\mu(f)=1$ then 
$$
	1=\mu(f)=\int_{A} f\,\dmu + \int_{K \setminus A} f\,\dmu 
	\stackrel{\eqref{setA}}{=}|\mu|(A)+\int_{K \setminus A} f\,\dmu
$$
and so $|\mu|(K \setminus A)=\int_{K \setminus A} f\,\dmu$. Since we have
$$
	\alpha:=\int\limits_{\{f \neq 1\}\cap P} f\,\dmu \leq |\mu|\big(\{f \neq 1\}\cap P\big),
	\quad \beta:=\int\limits_{\{f\neq -1\}\cap N} f\,\dmu \leq |\mu|\big(\{f \neq -1\}\cap N\big)
$$
and 
$$
	\alpha+\beta = \int\limits_{K \setminus A} f\,\dmu = 
	|\mu|(K\setminus A) = |\mu|\big(\{f \neq 1\}\cap P\big)+
	  |\mu|\big(\{f \neq -1\}\cap N\big),
$$
it follows that
\begin{equation}\label{null1}
	|\mu|\big(\{f \neq 1\}\cap P\big)=\alpha=\int\limits_{\{f \neq 1\}\cap P} f\,{\rm d}|\mu|
\end{equation}
and
\begin{equation}\label{null2}	
	|\mu|\big(\{f \neq -1\}\cap N\big)=\beta=-\int\limits_{\{f\neq -1\}\cap N} f\,{\rm d}|\mu|. 
\end{equation}
Clearly, \eqref{null1} yields $|\mu|\big(\{f \neq 1\}\cap P\big)=0$ and
\eqref{null2} yields $|\mu|\big(\{f \neq -1\}\cap N\big)=0$, so that
$|\mu|(K\setminus A)=0$. Therefore $|\mu|(A)=1$.
\end{proof}

\begin{defi}
Let $f\in C(K)$ and $0<\sigma< \ee$. Since the sets 
$\{f\geq 1-\sigma\}$ and $\{f\leq 1-\ee\}$ are closed and disjoint, 
Tietze extension theorem ensures the existence of a non-negative $u_{\sigma,\ee}^{f}\in B_{C(K)}$ 
such that 
$$
	u_{\sigma,\ee}^{f}|_{\{f\geq 1-\sigma\}}\equiv 1 
	\quad\mbox{and}\quad 
	u_{\sigma,\ee}^{f}|_{\{f\leq 1-\ee\}} \equiv 0.
$$
In the same way, there is a non-negative $v_{\sigma,\ee}^{f}\in B_{C(K)}$ 
such that 
$$
	v_{\sigma,\ee}^{f}|_{\{f\leq -1+\sigma\}}\equiv 1 
	\quad\mbox{and}\quad 
	v_{\sigma,\ee}^{f}|_{\{f\geq -1+\ee\}} \equiv 0.
$$
Given any $\mu\in \dualC$, we define $\mu_{\sigma,\ee}^{f,1},\mu_{\sigma,\ee}^{f,2}\in \dualC$ by
\[
	\mu_{\sigma,\ee}^{f,1}(g):=\int_K g\cdot u_{\sigma,\ee}^f\dmu
	\quad \mbox{and} \quad
	\mu_{\sigma,\ee}^{f,2}(g):=\int_K g\cdot v_{\sigma,\ee}^f\dmu \quad \mbox{for all } g\in C(K).
\]
\end{defi}

\begin{rem}
\begin{enumerate} 
\item If $\epsilon < 1$ then $\mu_{\sigma,\ee}^{f,1} \perp \mu_{\sigma,\ee}^{f,2}$.
\item The mappings $\mu \mapsto \mu_{\sigma,\ee}^{f,1}$ and $\mu\mapsto \mu_{\sigma,\ee}^{f,2}$ are  
$\weak^*$-$\weak^*$-continuous.
\end{enumerate}
\end{rem}
		
\begin{lem}\label{basiclemma}
Let $f\in B_{C(K)}$, $\mu\in B_{\dualC}$ and $0<\sigma<\ee < 1$. Then:
\begin{enumerate}
\item $\|\mu_{\sigma,\ee}^{f,1}\|\leq 1$ and $\|\mu_{\sigma,\ee}^{f,2}\|\leq 1$;
\item $\big\|(\mu_{\sigma,\ee}^{f,1})^+\big\| + \big\|(\mu_{\sigma,\ee}^{f,2})^-\big\| \geq 1-(1-\mu(f))/\sigma$;\label{condition1}
\item $\big\|(\mu_{\sigma,\ee}^{f,1})^-\big\| + \big\|(\mu_{\sigma,\ee}^{f,2})^+\big\| \leq (1-\mu(f))/\sigma$;\label{condition2}
\item $\big\| \mu - \mu_{\sigma,\ee}^{f,1} - \mu_{\sigma,\ee}^{f,2} \big\| \leq (1-\mu(f))/\sigma$.\label{condition3}
\end{enumerate}		
\end{lem}

\begin{proof}
Write $\mu_1:=\mu_{\sigma,\ee}^{f,1}$ and $\mu_2:=\mu_{\sigma,\ee}^{f,2}$. Let $(P,N)$
be a Hahn decomposition of~$\mu$ and define 
$$
	C:=\big(\{f\geq 1-\sigma\}\cap P\big) \cup \big(\{f\leq -1+\sigma\}\cap N\big).
$$
We claim that $|\mu|(C) \geq 1+(1-\mu(f))/\sigma$. Indeed, we have
\begin{equation}\label{ineqC}
	|\mu|(C) \geq \int_{C} f \dmu = \int_K f\dmu -\int_{K\setminus C} f \dmu 
	= \mu(f)-\int_{K \setminus C} f \dmu.
\end{equation}
Since 
\begin{align*}
	\int\limits_{K \setminus C} f \dmu&=
	\int\limits_{\{f<1-\sigma\}\cap P} f \dmu+
	\int\limits_{\{f>-1+\sigma\}\cap N} f \dmu= \\ &=
	\int\limits_{\{f<1-\sigma\}\cap P} f \ {\rm d}|\mu| +
	\int\limits_{\{f>-1+\sigma\}\cap N} (-f) \ {\rm d}|\mu| \leq (1-\sigma)|\mu|(K\setminus C),
\end{align*}
from~\eqref{ineqC} it follows that
$$
	|\mu|(C) \geq \mu(f)-(1-\sigma)(|\mu|(K)-|\mu|(C)) \geq 
	\mu(f)-(1-\sigma)(1-|\mu|(C)),
$$
which implies that $|\mu|(C)\geq 1-(1-\mu(f))/\sigma$, as claimed. 

(ii). Observe that $(P,N)$ is also a Hahn decomposition of $\mu_1$ and~$\mu_2$
(bear in mind that $u_{\sigma,\ee}^{f}\geq 0$ and $v_{\sigma,\ee}^{f}\geq 0$) and that
$C\cap P \subseteq \{f\geq 1-\sigma\}$ and $C\cap N \subseteq \{f\leq -1+\sigma\}$. Hence
$$
	\mu_1^+(C)=\mu_1(C\cap P)=
	\int_{C\cap P} u_{\sigma,\ee}^{f} \dmu=\mu(C\cap P)=|\mu|(C\cap P),
$$
$$
	\mu_2^-(C)=-\mu_2(C\cap N)=
	-\int_{C\cap N} v_{\sigma,\ee}^{f} \dmu=-\mu(C\cap N)=|\mu|(C\cap N),
$$
and therefore $\mu_1^+(C) + \mu^-_2(C) = |\mu|(C)$. We deduce that
$$
	\|\mu_1^+\|+\|\mu^-_2\| \geq \|\mu_1^+ + \mu^-_2\| \geq
	(\mu_1^+ + \mu^-_2)(C) = |\mu|(C) \geq 1-(1-\mu(f))/\sigma.
$$

(i) and (iii). Since $0 \leq u_{\sigma,\ee}^f+v_{\sigma,\ee}^f \leq 1$, we have
$\|\mu_1+\mu_2\| \leq \|\mu\|$. On the other hand, the equality $\|\mu_1+\mu_2\|=\|\mu_1\|+\|\mu_2\|$ holds because 
$\mu_1 \perp \mu_2$. Hence
\begin{multline*}
	1 \geq \|\mu\| \geq \|\mu_1\|+\|\mu_2\| = \\ = 
	\|\mu_1^+\|+\|\mu_1^-\|+\|\mu_2^+\|+\|\mu_2^-\|
	\stackrel{{\rm (i)}}{\geq} 1-(1-\mu(f))/\sigma+\|\mu_1^-\|+\|\mu_2^+\|,
\end{multline*}
which implies that $\|\mu_1^-\|+\|\mu_2^+\| \leq (1-\mu(f))/\sigma$.

(iv). Write $h:=1-u_{\sigma,\ee}^f-v_{\sigma,\ee}^f \in C(K)$, so that
$(\mu-\mu_1-\mu_2)(g)=\int_K g h \dmu$ for all $g\in C(K)$.
Since $0 \leq h \leq 1$ and $h$ vanishes on $C$, we get
$$
	\|\mu-\mu_1-\mu_2\|
	\leq |\mu|(K \setminus C) \leq 1- |\mu|(C) \leq (1-\mu(f))/\sigma,		
$$
which finishes the proof.
\end{proof}

\begin{lem}\label{lBPB-C(K)}
Suppose that $K$ admits local compensation. Let 
$f\in B_{C(K)} \setminus \{{\bf 0}\}$ and take  $1-\|f\|<\ee<1$. Then there exists
$f_0\in S_{C(K)}$ such that for every $F\in W(K)$ there is a $\weak^*$-continuous
function $\mathcal{P}_F\colon F^{-1}(\pi_2(f,\ee^2/6)) \to \pi_2(f_0)$ such that:
\begin{enumerate}
\item $\pi_2(f)\subseteq\pi_2(f_0)$ and $\norm{f-f_0}\leq\ee$;
\item $\norm{\mathcal{P}_F(t)-F(t)}\leq\ee$ for every $t\in F^{-1}(\pi_2(f,\ee^2/6))$. 
\end{enumerate}
\end{lem}

\begin{proof} 
We divide the proof into several steps.

\emph{Step 1.} Fix $\ee<\delta<1$. Note that $K$ is the union
of the following closed sets:
$$
	A:=\{f\geq 1-\ee\}, \quad
	B:=\{f\leq -1+\ee\}, \quad
	C:=\{-1+\delta \leq f\leq 1-\delta\},
$$
$$
	D:=\{1-\delta \leq f \leq 1-\ee\} \cup \{-1+\ee \leq f \leq -1+\delta\}.
$$
By Tietze extension theorem, there is a continuous function
$g\colon D \to [-\ee,\ee]$ such that
$$
	g|_{\{f=1-\ee\}}\equiv \ee, \quad
	g|_{\{f=1-\delta\}}\equiv 0, \quad
	g|_{\{f=-1+\ee\}}\equiv -\ee, \quad
	g|_{\{f=-1+\delta\}}\equiv 0.
$$
Now, we can define $f_0\in B_{C(K)}$ by declaring
$$
	f_0(t):=
	\begin{cases}
	1 & \text{if $t\in A$,}\\
	-1 & \text{if $t\in B$,}\\
	f(t) & \text{if $t\in C$,}\\
	f(t)+g(t) & \text{if $t\in D$.}
	\end{cases}
$$
It is straightforward that $\norm{f-f_0}\leq\ee$. Note also
that $A\cup B\neq \emptyset$ (because $\|f\|>1-\ee$) and so $\|f_0\|=1$.
To prove that $\pi_2(f) \subseteq \pi_2(f_0)$, suppose that
$\|f\|=1$, fix any $\mu \in \pi_2(f)$ and take a Hahn decomposition $(P,N)$ of~$\mu$. 
By Lemma~\ref{lemmapi} we have
$$
	|\mu|\big((\{f=1\}\cap P) \cup (\{f=-1\}\cap N)\big)=1.
$$
Since $\{f=1\} \subseteq \{f_0=1\}$ and $\{f=-1\} \subseteq \{f_0=-1\}$, another appeal
to Lemma~\ref{lemmapi} yields $\mu \in \pi_2(f_0)$.

\emph{Step 2.} Fix $F\in W(K)$. Set $\sigma:=5\ee/6$ and consider $F_1,F_2 \in W(K)$ defined by
$$
	F_1(t):=(F(t))_{\sigma,\ee}^{f,1}\quad\text{and}\quad
	F_2(t):=(F(t))_{\sigma,\ee}^{f,2}.
$$
Define now a $\weak^*$-continuous function $\mathcal{Q}\colon K \to \dualC$ by the formula
$$
	\mathcal{Q}(t):=\xi_{1}(t)-\xi_{2}(t),
$$
where $\xi_{1},\xi_{2} \in W(K)$ are compensations of $F_{1}$ and $-F_{2}$, respectively.

For every $t\in K$ we have 
$$
	\supp(\xi_{1}(t))\subseteq \supp(F_{1}(t))\subseteq A,\quad
	\supp(\xi_{2}(t))\subseteq \supp(-F_{2}(t))\subseteq B,
$$  
and $A\cap B=\emptyset$, hence $F_{1}(t) \perp F_{2}(t)$ and
$\xi_{1}(t) \perp \xi_{2}(t)$, and therefore
\begin{equation}
\begin{split}\label{lnormQ}
	1\geq \|F(t)\| \stackrel{(*)}{\geq} & 
	\norm{F_{1}(t)+F_{2}(t)} =\norm{F_{1}(t)}+\norm{F_{2}(t)} \geq \\ 
	\geq & \norm{\xi_{1}(t)}+\norm{\xi_{2}(t)}=
	\|\mathcal{Q}(t)\|=\xi_{1}(t)(K)+\xi_{2}(t)(K)
	\geq \\ \geq & F_{1}(t)(K)-F_{2}(t)(K).
\end{split}
\end{equation}
(inequality~$(*)$ was established in the proof of Lemma~\ref{basiclemma}(iii)).
It follows that 
\begin{equation}\label{lnormQe}
\begin{split}
	\langle \mathcal{Q}(t),f_0 \rangle &=
	\int_A f_0 \ {\rm d}\xi_{1}(t) - \int_B f_0 \ {\rm d}\xi_{2}(t)
					=\xi_{1}(t)(A) + \xi_{2}(t)(B) =\\
					&=\xi_{1}(t)(K) + \xi_{2}(t)(K)
					\stackrel{\eqref{lnormQ}}{=} \|\mathcal{Q}(t)\|.
\end{split}
\end{equation}
The $\weak^*$-continuity of~$\mathcal{Q}$ and~\eqref{lnormQe} imply that
the map $t \mapsto \|\mathcal{Q}(t)\|$ is continuous.

\emph{Step 3.} Fix $t\in K_0:=F^{-1}(\pi_2(f,\epsilon^2/6))$.
By Lemmas~\ref{simplifynorm} and~\ref{basiclemma}(iii), we have
\begin{multline}\label{laux1}
	\norm{\mathcal{Q}(t)-(F_1(t)+F_2(t))} \leq 
	\norm{\xi_{1}(t)-F_{1}(t)}
	+\norm{\xi_{2}(t)-(-F_{2}(t))} \leq \\ \leq
		2\big(\norm{(F_{1}(t))^-}+\norm{(-F_{2}(t))^-}\big)= 
		2\big(\norm{(F_{1}(t))^-}+\norm{(F_{2}(t))^+}\big) \leq \\
	\leq
	2\frac{(1-\la F(t),f \ra)}{\sigma} \stackrel{t\in K_0}{\leq}
	\frac{2\ee}{5}.
\end{multline}
On the other hand, by~\eqref{lnormQ} and Lemma~\ref{basiclemma}(ii)-(iii), we get
\begin{equation*}%\label{aux2}
\begin{split}
\norm{\mathcal{Q}(t)} &\geq F_{1}(t)(K)-F_{2}(t)(K)= \\  
	&=  \big(\norm{(F_1(t))^+}+\norm{(F_2(t))^-}\big)
	-\big(\norm{(F_{1}(t))^-}+ \norm{(F_{2}(t))^+}\big) \geq \\
	&\geq 1-2\frac{(1-\la F(t),f \ra)}{\sigma} \stackrel{t\in K_0}{\geq}  1-\frac{2\ee}{5}.
\end{split}
\end{equation*}
Hence $\mathcal{Q}(t)\neq 0$ and
\begin{equation}\label{laux2}
	\norm{\frac{\mathcal{Q}(t)}{\|\mathcal{Q}(t)\|}-\mathcal{Q}(t)}=
	1-\norm{\mathcal{Q}(t)} \leq \frac{2\ee}{5}
\end{equation}
(bear in mind that $\|\mathcal{Q}(t)\|\leq 1$, as shown in~\eqref{lnormQ}).
But Lemma~\ref{basiclemma}(iv) also yields
\begin{equation}\label{l3rdeq}
	\norm{F(t)-(F_{1}(t)+F_{2}(t))}\leq \frac{1-\la F(t),f\ra}{\sigma} \stackrel{t\in K_0}{\leq} \frac{\ee}{5}.
\end{equation}
Using \eqref{laux1}, \eqref{laux2} and \eqref{l3rdeq} we conclude that
\begin{multline*}
	\norm{\frac{\mathcal{Q}(t)}{\|\mathcal{Q}(t)\|}-F(t)} \leq\ \\ 
	 \leq \norm{\frac{\mathcal{Q}(t)}{\|\mathcal{Q}(t)\|}-\mathcal{Q}(t)}+
	\norm{\mathcal{Q}(t)-(F_{1}(t)+F_{2}(t))}
							+\norm{(F_{1}(t)+F_{2}(t))-F(t)} \leq \\ 
		\leq \frac{2\ee}{5}+\frac{2\ee}{5}+\frac{\ee}{5}= \ee.
\end{multline*}

\emph{Step 4.} The previous step makes clear that the function
$$
	\mathcal{P}_F\colon K_0 \to \dualC, \quad
	\mathcal{P}_F(t):=\frac{\mathcal{Q}(t)}{\|\mathcal{Q}(t)\|},
$$
is well-defined and satisfies $\|\mathcal{P}_F(t)-F(t)\|\leq \ee$ for every $t\in K_0$.
Note that \eqref{lnormQe} says that $\mathcal{P}_F(t) \in \pi_2(f_0)$ for every $t\in K_0$.
Since $\mathcal{Q}$ is $\weak^*$-continuous and the map $t \mapsto \|\mathcal{Q}(t)\|$
is continuous ({\em Step~2}), $\mathcal{P}_F$ is $\weak^*$-continuous as well.
The proof is over.
\end{proof}

The following particular case of the classical Bishop-Phelps-Bollob\'{a}s theorem
will be needed in the proof of Theorem~\ref{lBPBP-C(K)}.

\begin{cor}\label{lcoro-BPB-C(K)}
Suppose that $K$ admits local compensation. 
Let $(f,\mu)\in B_{C(K)}\times B_{\dualC}$ such that 
$\mu(f)\geq 1-\ee^2/6$, where $0<\ee <1$. Then 
there is $(f_0,\mu_0)\in \Pi(C(K))$ such that $\norm{f-f_0}\leq\ee$ and ${\norm{\mu-\mu_0}\leq\ee}$.
\end{cor}

\begin{proof}
Apply Lemma~\ref{lBPB-C(K)} to $f$ and the constant function $F\in W(K)$ given by
$F(t):=\mu$ for all $t\in K$, so that $F^{-1}(\pi_2(f,\ee^2/6))=K$. Then we can take any $\mu_0\in \mathcal{P}_F(K)$.
\end{proof}

\begin{rem}\label{lrem:Hahn}
In the situation of Lemma~\ref{lBPB-C(K)}, let $t\in F^{-1}(\pi_2(f,\ee^2/6))$. Then:
\begin{enumerate}
\item Every Hahn decomposition of~$F(t)$ is also a Hahn decomposition of~$\mathcal{P}_F(t)$. 
\item $\mathcal{P}_F(t) \ll F(t)$.
\end{enumerate}
\end{rem}

\begin{proof}
(i) Let $(P,N)$ be a Hahn decomposition of~$F(t)$. As we pointed out in the proof of Lemma~\ref{basiclemma}(ii), 
$(P,N)$ is a Hahn decomposition of both $F_{1}(t)$ and $F_{2}(t)$. We claim that 
for every Borel set $B \subseteq P$ we have
$\xi_{2}(t)(B)=0$.
Indeed, this is obvious if $F_{2}(t)(K)\geq 0$, while if $F_{2}(t)(K)< 0$ then 
$$
	0 \leq \xi_{2}(t)(B) \leq (-F_{2}(t))^+(B)=
	(F_{2}(t))^-(B)=F_{2}(t)(B\cap N)=0.
$$
Hence $\mathcal{Q}(t)(B)=\xi_{1}(t)(B)\geq 0$ for every Borel set $B \subseteq P$.
In the same way, we have $\mathcal{Q}(t)(B)=-\xi_{2}(t)(B)\leq 0$ for every Borel set $B \subseteq N$.

(ii) Obviously, $F_{1}(t) \ll F(t)$ and $F_{2}(t) \ll F(t)$. 
By the very definition of compensation, we also have $\xi_{1}(t)\ll F_{1}(t)$ and
$\xi_{2}(t) \ll F_{2}(t)$. Therefore $\mathcal{Q}(t) \ll F(t)$. 
\end{proof}

We are now ready to prove the main result of this section.

\begin{proof}[Proof of Theorem~\ref{lBPBP-C(K)}]
We shall check that if $K$ admits local compensation,
then $C(K)$ fulfills the requirements of Definition~\ref{defi:BPBpNR} with $\delta(\epsilon)=(\epsilon/6)^4$.
Let $T\in \mathfrak{L}(C(K))$ with $\nu(T)=1$ and $(f,\mu)\in \Pi(C(K))$ such that 
$\la \mu,T(f) \ra \geq 1-(\epsilon/6)^4$, where $0<\eps<1$.

\emph{Step 1.} By Corollary~\ref{lcoro-BPB-C(K)} applied to $(T(f),\mu)\in B_{C(K)}\times S_{\dualC}$
and $\delta:=\ee^2/7$ (note that $\la \mu,T(f)\ra\geq 1-\delta^2/6$), there is
$(g,\mu_0)\in\Pi(C(K))$ such that $\|T(f)-g\|\leq \delta$ and $\norm{\mu-\mu_0}\leq\delta<\ee$. 
Let $(P,N)$ be a Hahn decomposition of~$\mu$, which in turn is also a Hahn decomposition of~$\mu_0$ 
(see Remark~\ref{lrem:Hahn}(i)). Since $\mu(f)=1$, an appeal to Lemma~\ref{lemmapi} yields
$$
	|\mu|\big(K \setminus (\{f=1\}\cap P)\cup(\{f=-1\}\cap N)\big)=0.
$$
The fact that $\mu_0\ll \mu$ (see Remark~\ref{lrem:Hahn}(ii)) implies 
$$
	|\mu_0|\big(K \setminus (\{f=1\}\cap P)\cup(\{f=-1\}\cap N)\big)=0
$$
and so $\mu_0(f)=1$ (again, by Lemma~\ref{lemmapi}). Writing 
$$
	D_1:=\{T(f) \geq 1-\delta\} \quad \mbox{and} \quad
	D_2:=\{T(f) \leq -1+\delta\}, 
$$
the proof of Lemma~\ref{lBPB-C(K)} shows that
$\supp(\mu_0) \subseteq D_1\cup D_2$ and that 
$\mu_0(B)\geq 0$ (resp. $\mu_0(B)\leq 0$) for every Borel set $B \subseteq D_1$
(resp. $B \subseteq D_2$). Hence
\begin{equation}\label{lDs}
\mu_0(D_1)-\mu_0(D_2)=|\mu_0|(D_1 \cup D_2)=\norm{\mu_0}=1.
\end{equation}

\emph{Step 2.} Let us consider the closed sets
\begin{equation*}
\begin{split}
	A_1&:=\{T(f) \geq 1-\ee^2/6\} \supseteq D_1,\\
	A_2&:=\{T(f) \leq -1+\ee^2/6\} \supseteq D_2,\\
	C&:=\{-1+\ee^2/6 \leq T(f) \leq 1-\ee^2/6\}. 
\end{split}
\end{equation*}
Since $D_1\cap (C\cup A_2)=\emptyset=D_2 \cap (C \cup A_1)=\emptyset$, we can apply Tietze extension theorem to find
two continuous functions $g_1\colon K \to [0,1]$ and $g_2\colon K \to [-1,0]$ such that 
$$
		g_1|_{D_1}\equiv 1, \quad
		g_1|_{C\cup A_2}\equiv 0, \quad
		g_2|_{D_2}\equiv -1, \quad
		g_2|_{C\cup A_1}\equiv 0.
$$

\emph{Step 3.} Let $F,G\in W(K)$ be defined by $F(t):=T^*(\delta_t)=\delta_t\circ T$ and $G(t):=-F(t)$. 
It is clear that $F(A_1)\cup G(A_2) \subseteq \pi_2(f,\ee^2/6)$. By Lemma~\ref{lBPB-C(K)} there is $f_0\in S_{C(K)}$ such that 
$\pi_2(f) \subseteq \pi_2(f_0)$, $\norm{f-f_0}\leq\ee$ and there exist two $\weak^*$-continuous mappings 
$$
	\mathcal{P}_F\colon A_1 \to \pi_2(f_0) \quad\text{and}\quad
	\mathcal{P}_G\colon A_2 \to \pi_2(f_0)
$$ 
satisfying
\begin{equation}\label{lcontrolP}
	\sup_{t\in A_1}\norm{\mathcal{P}_F(t)-F(t)}\leq\ee
	\quad \text{and}\quad \sup_{t\in A_2}\norm{\mathcal{P}_{G}(t)+F(t)}\leq\ee.
\end{equation}
 
Now, we can define a $\weak^*$-continuous mapping $\widetilde{F}\colon K\to \dualC$ as follows:
$$
	\widetilde{F}(t):=
	\begin{cases}
	F(t) + g_1(t) \big(\mathcal{P}_F(t)-F(t)\big) & \text{if $t \in A_1$},\\
	F(t) + g_2(t) \big(\mathcal{P}_{G}(t)+F(t)\big) & \text{if $t \in A_2$},\\
	F(t) & \text{if $t \in C$}.
	\end{cases}
$$
Define $T_0\in\mathfrak{L}(C(K))$ by $T_0(h)(t):=\la \widetilde F(t),h \ra$ for every $h\in C(K)$ and $t\in K$. 
We shall check that $T_0$ satisfies the required properties.

\emph{Step 4.} Note that $\widetilde{F}(t)$ (resp. $-\widetilde{F}(t)$) is a convex combination
of $F(t)$ and $\mathcal{P}_F(t)$ (resp. $-F(t)$ and $\mathcal{P}_{G}(t)$) for every
$t\in A_1$ (resp. $t\in A_2$). Since $F(K) \subseteq B_{\dualC}$ 
and $\mathcal{P}_F(A_1)\cup\mathcal{P}_{G}(A_2)\subseteq \pi_2(f_0) \subseteq B_{\dualC}$, we deduce 
$\widetilde{F}(K) \subseteq B_{\dualC}$, which implies that
$$
	\norm{T_0}=
	\sup_{h\in B_{C(K)}}\norm{T_0(h)}=
	\sup_{h\in B_{C(K)}} \sup_{t\in K}|\la \widetilde F(t),h \ra|\leq 1.
$$
On the other hand,
$$
	\norm{T_0-T}=
	\sup_{h\in B_{C(K)}} \sup_{t\in K}|\la \widetilde{F}(t)-F(t),h \ra| 
	\leq \sup_{t\in K} \big\|\widetilde{F}(t)-F(t)\big\|\stackrel{\eqref{lcontrolP}}{\leq}\ee.
$$

Since $(f,\mu_0)\in\Pi(C(K))$ (as shown in {\em Step~1}) and $\pi_2(f)\subseteq\pi_2(f_0)$, we deduce that 
$(f_0,\mu_0)\in\Pi(C(K))$. 
Since $g_1|_{D_1}\equiv 1$, $g_2|_{D_2}\equiv -1$ and $\mathcal{P}_F(A_1)\cup\mathcal{P}_G(A_2) \subseteq \pi_2(f_0)$, we have
$$
	T_0(f_0)(t)=
	\begin{cases}
	\la \mathcal{P}_F(t),f_0 \ra =1 &\text{if $t\in D_1$,}\\
	- \la \mathcal{P}_{G}(t),f_0 \ra  =-1 &\text{if $t\in D_2$.}\\
	\end{cases}
$$
Bearing in mind that $\supp(\mu_0) \subseteq D_1\cup D_2$ (as pointed out in {\em Step~1}),
it follows that
\begin{multline*}
	\la \mu_0, T_0(f_0) \ra=\int_{D_1\cup D_2} T_0(f_0) \, d\mu_0 = \\ =
	\int_{D_1} T_0(f_0) \, d\mu_0+\int_{D_2} T_0(f_0) \, d\mu_0=
	\mu_0(D_1)-\mu_0(D_2)\stackrel{\eqref{lDs}}{=}1.
\end{multline*}
In particular, this implies that $\nu(T_0)=1$. The proof is over.
\end{proof}

\section{Existence of compensation functions for metric compacta}\label{section3}

This section is devoted to proving that every compact metric space $K$ admits local compensation.
Actually, we shall show that a stronger property holds true, namely, that 
every $F\in W(K)$ admits a compensation of the form $\xi\circ F$, where 
$\xi\colon \dualC \to \dualC$ is a function (depending only on~$K$) as in the following definition:

\begin{defi}\label{defi:MCompensation}
Let $K$ be a compact space and $M \subseteq \dualC$. We say that $\xi\colon M \to \dualC$ is an \emph{$M$-compensation function} if 
it is $\weak^*$-$\weak^*$-continuous and $\xi(\mu)$ is a compensation of~$\mu$ 
for every $\mu\in M$; if in addition $M=\dualC$, we say that $\xi$ is a {\em compensation function}.
\end{defi}

Thus, in this section our goal is to prove the following:

\begin{theo}\label{theo:CompactMetric}
Every compact metric space admits a compensation function.
\end{theo}

\begin{cor}
If $K$ is a compact metric space, then $C(K)$ has the BPB property for numerical radius.
\end{cor}
\begin{proof}
Combine Theorems~\ref{lBPBP-C(K)} and~\ref{theo:CompactMetric}.
\end{proof}

\begin{cor}\label{cor:LEMITA}
Let $T$ be a topological space, $K$ a compact space and $F\colon T \to \dualC$
a $\weak^*$-continuous function. Suppose there is a compact metrizable set
$L \subseteq K$ such that $\supp(F(t)) \subseteq L$ for every $t\in T$.
Then there is a $w^*$-continuous function $G\colon T \to \dualC$ such that
$G(t)$ is a compensation of~$F(t)$ for every $t\in T$.
\end{cor}
\begin{proof}
According to Theorem~\ref{theo:CompactMetric}, $L$ admits a compensation function $\xi\colon \mathcal{M}(L) \to \mathcal{M}(L)$.
Let $U\colon \dualC \to \mathcal{M}(L)$ and $V\colon C(K) \to C(L)$
be the restriction operators. Since $\supp(F(t)) \subseteq L$ for every $t\in T$,
 the composition $U\circ F$ is $\weak^*$-continuous. It is now clear that 
$G:=V^*\circ \xi \circ U\circ F$ satisfies the required properties.
\end{proof}

In order to prove Theorem~\ref{theo:CompactMetric} we need some previous work.
Given a continuous onto mapping $\varphi\colon K \to L$ between compact spaces, let 
$C_\varphi\colon C(L)\to C(K)$ be the operator
defined by $C_\varphi(f):=f\circ \varphi$ for every $f\in C(L)$.
An operator $u\colon C(K)\to C(L)$ is called a \emph{regular averaging operator} for~$\varphi$ provided that 
$u$ is positive, $u(\mathbbm{1}_K)=\mathbbm{1}_L$ and $u\circ C_\varphi=id_{C(L)}$.

\begin{lem}\label{pro:QuotientCompensation}
Let $K$ and $L$ be compact spaces for which there is
a continuous onto mapping $\varphi\colon K \to L$ with a regular averaging operator.
If $K$ admits a compensation function, then $L$ admits a compensation function as well.
\end{lem}

\begin{proof}
Let $\xi\colon\mathcal{M}(K)\to\mathcal{M}(K)$ be a compensation function and 
$u\colon C(K)\to C(L)$ a regular averaging operator for~$\varphi$. 
Define
$$
	\widetilde{\xi}\colon\mathcal{M}(L)\to\mathcal{M}(L), \quad 
	\widetilde{\xi}:=C_\varphi^{*}\circ \xi \circ u^*.
$$
Clearly, $\widetilde{\xi}$ is $\omega^*$-$\omega^*$-continuous. Fix $\mu \in \mathcal{M}(L)$. Since
$C_\varphi$ is a positive operator, so is $C_\varphi^*$ and therefore
$\widetilde{\xi}(\mu) \geq 0$. Since
$$
	u^*(\mu)(K)=\la u^*(\mu), \mathbbm{1}_K \ra= \la \mu, u(\mathbbm{1}_K) \ra=\la \mu,\mathbbm{1}_L \ra= \mu(L)
$$ 
and
$$
	\widetilde{\xi}(\mu)(L)=\langle \widetilde{\xi}(\mu),\mathbbm{1}_L \rangle=
	\la \xi(u^*(\mu)), C_\varphi(\mathbbm{1}_L) \ra=
	\la \xi(u^*(\mu)), \mathbbm{1}_K \ra= \xi(u^*(\mu))(K),
$$
we deduce that $\widetilde{\xi}(\mu)=0$ if $\mu(L)\leq 0$
and $\widetilde{\xi}(\mu)(L)=\mu(L)$ if $\mu(L)>0$. For every non-negative $f\in C(K)$ we have
\begin{multline*}
	\big\la \widetilde{\xi}(\mu), f \big\ra=
	\big\la \xi(u^*(\mu)), C_\varphi(f) \big\ra \leq
	\big\la (u^*(\mu))^+, C_\varphi(f) \big\ra  \leq 
	\big\la u^*(\mu^+), C_\varphi(f) \big\ra = \\ = \big\la \mu^+, u(C_\varphi(f)) \big\ra = \big\la \mu^+, f \big\ra,
\end{multline*}
because $C_\varphi$ and $u^*$ are positive operators and $u\circ C_\varphi=id_{C(L)}$.
Hence $\widetilde{\xi}(\mu) \leq \mu^+$.
It follows that $\widetilde{\xi}$ is a compensation function.
\end{proof}

From now on we write $\cantor:=2^\N=\{0,1\}^{\N}$ to denote the Cantor set.
Pe\l czynski proved that a compact space $L$ is metrizable if, and only if, 
there is a continuous onto mapping $\varphi\colon\mathcal{C}\to L$ with a regular averaging operator,
\cite[Theorem~5.6]{pel-J}.
This result and Lemma~\ref{pro:QuotientCompensation} show that Theorem~\ref{theo:CompactMetric}
can be deduced from the following particular case:

\begin{theo}\label{theo:Cantor}
The Cantor set $\cantor$ admits a compensation function.
\end{theo}

Such compensation function will be defined explicitly (Definition~\ref{defi:CompensationCantor}
and Proposition~\ref{pro:CompensationCantor}).
The rest of this section is devoted to proving Theorem~\ref{theo:Cantor}. We divide the proof
into three subsections for the convenience of the reader. We first need to introduce
some notation.

\begin{defi}
We define a continuous function $d\colon \R \times \R \to \R$ by 
$$
	d(s_1,s_2):=
	\begin{cases}
		{\rm sign}(s_2)\cdot\min\{|s_1|,|s_2|\} & \text{if $s_1\cdot s_2 <0$,}\\
		0 & \text{if $s_1\cdot s_2 \geq 0$}.
	\end{cases}
$$
\end{defi}

\begin{rem}\label{lem:d}
The function $d$ satisfies the following properties:
\begin{enumerate}
\item $d(s_1,s_2)=-d(s_2,s_1)$.
\item $0 \leq 1+d(s_1,s_2)/{s_1} \leq 1$ if $s_1 \neq 0$.
\item $0 \leq s_1+d(s_1,s_2) \leq s_1$ if $s_1\geq 0$.
\item $s_1 \leq s_1+d(s_1,s_2) \leq 0$ if $s_1\leq 0$.
\item If $s_1\cdot s_2<0$ then either $s_1+d(s_1,s_2)=0$ or $s_2-d(s_1,s_2)=0$.
\end{enumerate}
\end{rem}

As usual, we write $\tcantor$ to denote the set of all finite (maybe empty) sequences
of $0$s and $1$s. Given $\sigma=(\sigma(1),\dots,\sigma(n))\in \tcantor$, we write 
${\rm length}(\sigma)=n$ and 
$$
	\sigma|_{k}:=(\sigma(1),\dots,\sigma(k)) \in \tcantor \quad\mbox{for every }k\in \{1,\dots,n\};
$$ 
we use the convention $\sigma|_0=\emptyset$. We denote
$$
	\sigma\smallfrown 0:=(\sigma(1),\dots,\sigma(n),0) \quad\mbox{and}\quad 
	\sigma\smallfrown 1:=(\sigma(1),\dots,\sigma(n),1).
$$
More generally, if $\tau=(\tau(1),\dots,\tau(m))\in \tcantor$, we write
$$
	\sigma \smallfrown \tau :=(\sigma(1),\dots,\sigma(n),\tau(1),\dots,\tau(m)).
$$
Given any $\sigma'\in \tcantor$, the notation $\sigma \subseteq \sigma'$ means
that ${\rm length}(\sigma')\geq n$ and $\sigma'(k)=\sigma(k)$ for every $k\in \{1,\dots,n\}$.
Analogously, given any $t=(t(k))_{k\in \N}$ in the Cantor set~$\cantor$, the notation $\sigma \subseteq t$ means
that $t(k)=\sigma(k)$ for every $k\in \{1,\dots,n\}$.
Thus, the standard clopen basis for the topology of $\cantor$ consists of the sets
\[
	N_\sigma:=\{t\in\cantor\colon \, \sigma\subseteq t\}, \quad \sigma\in 2^{<\mathbb{N}}.
\]
For every $n\in \N \cup \{0\}$ we write $\cantor_n:=\{\sigma\in\tcantor\colon\,{\rm length}(\sigma)=n\}$.

\subsection{Construction}\label{subsection:construction}

Fix $\mu \in \dualC[\cantor]$ and let $m_\sigma:=\mu(N_\sigma)$ for every $\sigma\in\tcantor$. We 
next define a collection of real numbers $\{\tilde{m}_{\sigma}\colon\sigma \in \tcantor\}$ satisfying some special
properties which shall be discussed in Subsection~\ref{subsection:properties}. 

Fix $n\in \N \cup \{0\}$. In order to define the collection $\{\tilde{m}_{\sigma}\colon\sigma \in \cantor_n\}$,  
we construct certain real numbers $\{\mm{k}{\sigma}\colon\,\sigma \in \cantor_n\}$ for every $k\in \{0,\dots,n\}$. This is done 
inductively:

\begin{itemize}
\item Case $k=0$. Set $\mm{0}{\sigma}:=m_\sigma$ for every $\sigma \in \cantor_n$.

\item Case $k=1$. For each $\tau \in \cantor_{n-1}$ we set
\[\begin{split}
\mm{1}{\tau \smallfrown 0} &:= m_{\tau \smallfrown 0} + d(m_{\tau \smallfrown 0},m_{\tau \smallfrown 1}),\\
\mm{1}{\tau \smallfrown 1} &:= m_{\tau \smallfrown 1} - d(m_{\tau \smallfrown 0},m_{\tau \smallfrown 1}).
\end{split}\]

\item Assume that $k \in \{2,\dots,n\}$ and that the collection
$\{\mm{k-1}{\sigma}\colon\,\sigma \in \cantor_n\}$ is already constructed. 
Note that $\cantor_n$ is the disjoint union of the sets 
$$
	\cantor_{n,\tau}:=\{\sigma\in \cantor_{n}\colon \, \tau \subseteq \sigma\}, \quad
	\tau\in \cantor_{n-k}.
$$
Fix $\tau \in \cantor_{n-k}$. We define
\[
	s_{n,\tau,0}:=\sum_{\sigma\in \cantor_{n,\tau,0}} \mm{k-1}{\sigma}
	\quad \text{ and } 
	\quad s_{n,\tau,1}:=\sum_{\sigma\in \cantor_{n,\tau,1}} \mm{k-1}{\sigma},
\]
where
$\cantor_{n,\tau,0}:=\{\sigma\in \cantor_{n,\tau}\colon \sigma(n-k+1)=0\}$ and 
$\cantor_{n,\tau,1}:=\cantor_{n,\tau} \setminus \cantor_{n,\tau,0}$.
We now distinguish two cases:
\begin{itemize}
\item If $s_{n,\tau,0}\cdot s_{n,\tau,1} = 0$, then we set
$$
	\mm{k}{\sigma}:=\mm{k-1}{\sigma} \quad \mbox{for every }\sigma \in \cantor_{n,\tau}.
$$
\item If $s_{n,\tau,0}\cdot s_{n,\tau,1} \neq 0$, then we set
\[\begin{split}
	\mm{k}{\sigma} &:= \mm{k-1}{\sigma}\cdot \Big(1 + \frac{d(s_{n,\tau,0},s_{n,\tau,1})}{s_{n,\tau,0}}\Big)
	\quad \text{ for every } \sigma \in \cantor_{n,\tau,0},\\
	\mm{k}{\sigma} &:= \mm{k-1}{\sigma}\cdot \Big(1 - \frac{d(s_{n,\tau,0},s_{n,\tau,1})}{s_{n,\tau,1}} \Big)
	\quad \text{ for every } \sigma \in \cantor_{n,\tau,1}.
\end{split}\]
\end{itemize}
In this way, the collection $\{\mm{k}{\sigma}\colon\,\sigma \in \cantor_n\}$ is constructed.
\end{itemize}

Finally, we define $\tilde{m}_{\sigma}:=\mm{n}{\sigma}$ for every $\sigma \in \cantor_n$ and $n\in \N\cup \{0\}$.

\subsection{Properties}\label{subsection:properties}

Fix $\mu\in \dualC[\cantor]$. We follow the notations introduced in Subsection~\ref{subsection:construction}.

\begin{lem}\label{lemmaConserva}
$\tilde{m}_\sigma =\tilde{m}_{\sigma \smallfrown 0}+\tilde{m}_{\sigma \smallfrown 1}$ for every $\sigma\in\tcantor$.
\end{lem}

\begin{proof} Fix $n\in \N \cup \{0\}$. We shall prove that 
$$
	\mm{k}{\sigma}=\mm{k+1}{\sigma \smallfrown 0}+\mm{k+1}{\sigma \smallfrown 1}
	\quad
	\mbox{for every }\sigma \in \cantor_n \mbox{ and }k\in \{0,1,\dots,n\}
$$
by induction on~$k$. Note that for $k=0$ we have
$$
	\mm{0}{\sigma}=\mf{\sigma}=\mu(N_\sigma)=
	\mu(N_{\sigma \smallfrown 0})+\mu(N_{\sigma \smallfrown 1})=
	\mf{\sigma \smallfrown 0}+\mf{\sigma \smallfrown 1}=\mm{1}{\sigma \smallfrown 0}+\mm{1}{\sigma \smallfrown 1}
$$
for every $\sigma \in \cantor_n$, by the very definition of $\mm{1}{\sigma \smallfrown 0}$ and $\mm{1}{\sigma \smallfrown 1}$.
Suppose that $k\in \{1,\dots,n\}$ and that the inductive hypothesis holds:
\begin{equation}\label{induction}
	\mm{k-1}{\sigma'}=\mm{k}{\sigma' \smallfrown 0}+\mm{k}{\sigma' \smallfrown 1}
	\quad\mbox{for every }\sigma'\in \cantor_n.
\end{equation}
Fix $\sigma \in \cantor_n$ and let $\tau:=\sigma|_{n-k}$, so that
\begin{multline*}
	s_{n+1,\tau,0}=\sum_{\sigma'\in \cantor_{n+1,\tau,0}} \mm{k}{\sigma'}=
	\sum_{\sigma'\in \cantor_{n,\tau,0}} \mm{k}{\sigma' \smallfrown 0}+\sum_{\sigma'\in \cantor_{n,\tau,0}}\mm{k}{\sigma' \smallfrown 1} 
	=\\= \sum_{\sigma'\in \cantor_{n,\tau,0}} \big( \mm{k}{\sigma' \smallfrown 0}+\mm{k}{\sigma' \smallfrown 1} \big)
	\stackrel{\eqref{induction}}{=}
	\sum_{\sigma'\in \cantor_{n,\tau,0}} \mm{k-1}{\sigma'}=s_{n,\tau,0}.
\end{multline*}
In the same way, we have $s_{n+1,\tau,1}=s_{n,\tau,1}$. 

If $s_{n+1,\tau,0}\cdot s_{n+1,\tau,1}=0$, then 
$\mm{k+1}{\sigma \smallfrown 0}=\mm{k}{\sigma \smallfrown 0}$, $\mm{k+1}{\sigma \smallfrown 1}=\mm{k}{\sigma \smallfrown 1}$
and $\mm{k}{\sigma}=\mm{k-1}{\sigma}$, hence 
$$
	\mm{k+1}{\sigma \smallfrown 0}+\mm{k+1}{\sigma \smallfrown 1}=
	\mm{k}{\sigma \smallfrown 0}+\mm{k}{\sigma \smallfrown 1}\stackrel{\eqref{induction}}{=}\mm{k-1}{\sigma}=\mm{k}{\sigma}.
$$

If $s_{n+1,\tau,0}\cdot s_{n+1,\tau,1}\neq 0$, then
\begin{multline*}
	\mm{k+1}{\sigma \smallfrown 0}+\mm{k+1}{\sigma \smallfrown 1} = \\
	=\big(\mm{k}{\sigma \smallfrown 0}+\mm{k}{\sigma \smallfrown 1}\big) \cdot \Big(1 + (-1)^{\sigma(n-k+1)}
	\frac{d(s_{n+1,\tau,0},s_{n+1,\tau,1})}{s_{n+1,\tau,\sigma(n-k+1)}}\Big)
	=\\ \stackrel{\eqref{induction}}{=}
	\mm{k-1}{\sigma} \cdot\Big(1 + (-1)^{\sigma(n-k+1)}
	\frac{d(s_{n,\tau,0},s_{n,\tau,1})}{s_{n,\tau,\sigma(n-k+1)}}\Big)
	=\mm{k}{\sigma},
\end{multline*}
which finishes the proof.
\end{proof}

\begin{lem}\label{lem:sumatotal}
$\mu(\cantor)=\sum_{\sigma\in \cantor_n} \tilde{m}_{\sigma}$ for every $n\in \N \cup \{0\}$.
\end{lem}

\begin{proof}
By induction on~$n$. The case $n=0$ is obvious. 
Suppose $n>0$ and the inductive hypothesis. By applying Lemma~\ref{lemmaConserva}
we get
\begin{multline*}
	\sum_{\tau\in \cantor_n} \tilde{m}_{\tau}=
	\sum_{\sigma\in \cantor_{n-1}} \tilde{m}_{\sigma \smallfrown 0}+\sum_{\sigma\in \cantor_{n-1}} \tilde{m}_{\sigma \smallfrown 1}=
	\sum_{\sigma\in \cantor_{n-1}} \big(\tilde{m}_{\sigma \smallfrown 0}+\tilde{m}_{\sigma \smallfrown 1}\big)= \\ =
	\sum_{\sigma\in \cantor_{n-1}} \tilde{m}_{\sigma}=
	\mu(\cantor),
\end{multline*}
as required.
\end{proof}

\begin{lem}\label{lem:signs}
Let $n\in \N$, $k\in \{1,\dots,n\}$ and $\tau \in \cantor_{n-k}$. Then the collection
$$
	\{\mm{k}{\tau \smallfrown \tau'}\colon \tau'\in \cantor_k\}
$$ 
has constant sign. In particular, $\{\tilde{m}_{\sigma}\colon \sigma\in \cantor_n\}$
has constant sign.
\end{lem}

\begin{proof} We proceed by induction on~$k$. The case $k=1$ follows
immediately from Remark~\ref{lem:d}(v). Suppose that $k\in \{2,\dots,n\}$
and that the inductive hypothesis holds. Define $\tau_i:=\tau \smallfrown i$ for $i\in \{0,1\}$, so that
$\tau_i\in \cantor_{n-k+1}$ and each of the collections
$$
	\{\mm{k-1}{\tau_i \smallfrown \tau'}\colon  \tau'\in \cantor_{k-1}\}
$$
has constant sign, which in turn coincides with the sign of
$$
	s_{n,\tau,i}=\sum_{\sigma\in \cantor_{n,\tau,i}} \mm{k-1}{\sigma}=
	\sum_{\tau'\in \cantor_{k-1}} \mm{k-1}{\tau_i \smallfrown \tau'}.
$$
If $s_{n,\tau,0}\cdot s_{n,\tau,1}\geq 0$, then 
$\mm{k}{\tau \smallfrown \tau'}=\mm{k-1}{\tau \smallfrown \tau'}$
for every $\tau'\in \cantor_k$ and so the collection
$$
	\{\mm{k}{\tau \smallfrown \tau'}\colon \tau'\in \cantor_{k}\}=
	\{\mm{k-1}{\tau_0 \smallfrown \tau'}\colon \tau'\in \cantor_{k-1}\}\cup
	\{\mm{k-1}{\tau_1 \smallfrown \tau'}\colon \tau'\in \cantor_{k-1}\}
$$
has constant sign, as required.
If $s_{n,\tau,0}\cdot s_{n,\tau,1}< 0$, then 
\begin{equation}\label{duplo}
\begin{split}
	\mm{k}{\tau_0 \smallfrown \tau'} &:= \mm{k-1}{\tau_0 \smallfrown \tau'} \cdot \Big(1+ 
	\frac{d(s_{n,\tau,0},s_{n,\tau,1})}{s_{n,\tau,0}} \Big)\\
	\mm{k}{\tau_1 \smallfrown \tau'} &:= \mm{k-1}{\tau_1 \smallfrown \tau'} \cdot \Big(1-
		\frac{d(s_{n,\tau,0},s_{n,\tau,1})}{s_{n,\tau,1}} \Big)
\end{split}
\end{equation}
for every $\tau' \in \cantor_{k-1}$, so 
each of the collections $\{\mm{k}{\tau_i \smallfrown \tau'}\colon \tau'\in \cantor_{k-1}\}$ has constant sign. 
On the other hand, by Remark~\ref{lem:d}(v) and~\eqref{duplo}, we have
either $\mm{k}{\tau_0 \smallfrown \tau'}=0$ for every $\tau' \in \cantor_{k-1}$
or $\mm{k}{\tau_1 \smallfrown \tau'}=0$ for every $\tau' \in \cantor_{k-1}$. It follows that
the collection
$$
	\{\mm{k}{\tau \smallfrown \tau'}\colon \tau'\in \cantor_{k}\}=
	\{\mm{k}{\tau_0 \smallfrown \tau'}\colon \tau'\in \cantor_{k-1}\}\cup
	\{\mm{k}{\tau_1 \smallfrown \tau'}\colon \tau'\in \cantor_{k-1}\}
$$
has constant sign and the proof is over.
\end{proof}

\begin{lem}\label{lemmaconstructo}
Let $\sigma\in\tcantor$. The following statements hold:
\begin{enumerate}
\item If $m_\sigma\geq 0$, then $0\leq\tilde{m}_\sigma \leq m_\sigma$.
\item If $\mu(\cantor) \geq 0$ and $m_\sigma\leq 0$, then $\tilde{m}_\sigma=0$.
\end{enumerate}
\end{lem}
\begin{proof}
Write $n:=\mathrm{length}(\sigma)$. 

(i) We shall prove that 
\begin{equation}\label{chainPOS}
	0\leq \mm{n}{\sigma} \leq \mm{n-1}{\sigma} \leq \dots \leq \mm{1}{\sigma} \leq m_{\sigma}.
\end{equation}
The inequalities $0\leq \mm{1}{\sigma}\leq m_\sigma$ follow immediately
from the very definition of $\mm{1}{\sigma}$ and Remark~\ref{lem:d} (parts (i) and~(iii)). Assume that 
$0\leq\mm{k-1}{\sigma}\leq\dots\leq\mm{1}{\sigma}\leq m_\sigma$ for some $k\in \{2,\dots,n\}$. 
Write $\tau:=\sigma|_{n-k}$. If $s_{n,\tau,0}\cdot s_{n,\tau,1} \geq 0$, then $\mm{k}{\sigma}=\mm{k-1}{\sigma}$;
otherwise we have
$$
	\mm{k}{\sigma}=
	\begin{cases}
	\mm{k-1}{\sigma}\cdot \Big(1+ \frac{d(s_{n,\tau,0},s_{n,\tau,1})}{s_{n,\tau,0}}\Big) & \text{if $\sigma(n-k+1)=0$}\\
	\mm{k-1}{\sigma}\cdot \Big(1- \frac{d(s_{n,\tau,0},s_{n,\tau,1})}{s_{n,\tau,1}}\Big) & \text{if $\sigma(n-k+1)=1$}
	\end{cases}
$$
and in either case $0\leq\mm{k}{\sigma}\leq\mm{k-1}{\sigma}$ (by Remark~\ref{lem:d} --(i) and~(ii)-- 
bearing in mind that $\mm{k-1}{\sigma}\geq 0$). This proves~\eqref{chainPOS} and therefore
$0\leq \mm{n}{\sigma}=\tilde{m}_\sigma \leq m_\sigma$.

(ii) In the same way, the following chain of inequalities holds true: 
\begin{equation*}
	0\geq \tilde{m}_\sigma=\mm{n}{\sigma} \geq \mm{n-1}{\sigma} \geq \dots \geq \mm{1}{\sigma} \geq m_{\sigma}.
\end{equation*}
We now argue by contradiction. Suppose that $\tilde{m}_\sigma<0$. Then Lemma~\ref{lem:signs} ensures that $\tilde{m}_{\sigma'}\leq 0$
for every $\sigma' \in \cantor_n$.  
Bearing in mind Lemma~\ref{lem:sumatotal}, we obtain
$$
	0\leq \mu(\cantor)=\sum_{\sigma' \in \cantor_n}\tilde{m}_{\sigma'} \leq \tilde{m}_{\sigma}< 0,
$$
a contradiction. The proof is over.
\end{proof}

\subsection{Compensation function}
We follow the notations introduced in Subsection~\ref{subsection:construction} with
some obvious modifications to denote dependence with respect to~$\mu\in \dualC[\cantor]$.

\begin{defi}\label{defi:CompensationCantor}
Let $\mu\in \dualC[\cantor]$. We define $\xi(\mu) \in \mathcal{M}(\cantor)$ as follows:
\begin{enumerate}
\item If $\mu(\cantor)<0$, then $\xi(\mu):=0$.
\item If $\mu(\cantor)\geq 0$, then $\xi(\mu)$ is the unique element of~$\dualC[\cantor]$ such that 
$$
	\xi(\mu)(N_\sigma)=\tilde{m}_\sigma(\mu) \quad \mbox{ for every }\sigma\in\tcantor.
$$
The existence of $\xi(\mu)$ is ensured by Lemma~\ref{lemmaConserva}
via a standard argument.
\end{enumerate}
\end{defi}

\begin{pro}\label{pro:CompensationCantor}
$\xi\colon \mathcal{M}(\cantor) \to \mathcal{M}(\cantor)$ is a compensation function.
\end{pro}

\begin{proof}
Given any $\mu \in \dualC[\cantor]$ with $\mu(\cantor)\geq 0$, we have
$0 \leq \xi(\mu)(N_\sigma) \leq \mu^+(N_\sigma)$ 
for every $\sigma\in\tcantor$ (thanks to Lemma~\ref{lemmaconstructo})
and, by the very definitions, $\xi(\mu)(\cantor)=\mu(\cantor)$. 
Hence $\xi(\mu)$ is a compensation of~$\mu$ for every $\mu\in \dualC[\cantor]$.

To prove that $\xi$ is a compensation function, it only remains to show that it 
is $\weak^*$-$\weak^*$-continuous. Of course,
it suffices to check the continuity of~$\xi$ on 
$$
	\mathcal{H}:=\{\mu \in \mathcal{M}(\cantor)\colon \, \mu(\cantor)\geq 0\},
$$
which is equivalent to saying that, for every $\sigma \in \tcantor$, the real-valued function 
$$
	\mu \mapsto \xi(\mu)(N_\sigma)=\tilde{m}_\sigma(\mu)
$$ 
is $\weak^*$-continuous on~$\mathcal{H}$. Fix $n\in \N\cup \{0\}$. We shall prove that 
$$
	\mu \mapsto \mm{k}{\sigma}(\mu) \quad
	\mbox{is }\weak^*\mbox{-continuous on }\mathcal{H}
	\mbox{ for every }\sigma \in \cantor_n \mbox{ and }k\in \{0,1,\dots,n\}
$$
by induction on~$k$. The cases $k=0$ and $k=1$ are obvious. Suppose $k\in \{2,\dots,n\}$ and that the inductive hypothesis
holds.
Fix $\sigma \in \cantor_n$ and write $\tau:=\sigma|_{n-k}$. Then the mappings
\[
	s_{n,\tau,0}(\cdot)=\sum_{\sigma'\in \cantor_{n,\tau,0}} \mm{k-1}{\sigma'}(\cdot)
	\quad \text{ and } 
	\quad s_{n,\tau,1}(\cdot)=\sum_{\sigma'\in \cantor_{n,\tau,1}} \mm{k-1}{\sigma'}(\cdot)
\]
are $\weak^*$-continuous on~$\mathcal{H}$. Suppose that $\sigma(n-k+1)=0$ (the other case is analogous).
Then for every $\mu \in \mathcal{H}$ we have
\begin{equation}\label{laformula}
	\mm{k}{\sigma}(\mu)= \mm{k-1}{\sigma}(\mu)\cdot \Big(1+\frac{d(s_{n,\tau,0}(\mu),s_{n,\tau,1}(\mu))}{s_{n,\tau,0}(\mu)}\Big)
\end{equation}
if $s_{n,\tau,0}(\mu)\neq 0$, while $\mm{k}{\sigma}(\mu)= \mm{k-1}{\sigma}(\mu)$ if $s_{n,\tau,0}(\mu)=0$.
From~\eqref{laformula} it follows
at once that $\mm{k}{\sigma}(\cdot)$ is $\weak^*$-continuous at every $\mu \in \mathcal{H}$ with $s_{n,\tau,0}(\mu)\neq 0$.

Take any $\mu_0 \in \mathcal{H}$ with $s_{n,\tau,0}(\mu_0)=0$. Since 
$\{\mm{k-1}{\sigma'}(\mu_0)\colon \sigma'\in \cantor_{n,\tau,0}\}$ has constant sign
(by Lemma~\ref{lem:signs} applied to~$\tau\smallfrown 0$ and $k-1$), we get $\mm{k-1}{\sigma}(\mu_0)=0$ and so
$\mm{k}{\sigma}(\mu_0)=\mm{k-1}{\sigma}(\mu_0)=0$. Bearing in mind~\eqref{laformula}
and Remark~\ref{lem:d}(ii), we obtain
$$
	\big|\mm{k}{\sigma}(\mu)-\mm{k}{\sigma}(\mu_0)\big|
	=\big|\mm{k}{\sigma}(\mu)\big| \leq \big|\mm{k-1}{\sigma}(\mu)\big|=
	\big|\mm{k-1}{\sigma}(\mu)-\mm{k-1}{\sigma}(\mu_0)\big|
$$
for every $\mu\in \mathcal{H}$. This inequality and the inductive hypothesis imply
that the mapping $\mm{k}{\sigma}(\cdot)$ is $\weak^*$-continuous at~$\mu_0$. The proof is finished.
\end{proof}

\section{Beyond the metrizable case}\label{section4}

In this section we discuss the existence of compensation functions
in certain non-metrizable compacta. Specifically, we deal
with one-point compactifications of discrete sets 
(Subsection~\ref{subsection:compactification})
and ordinal intervals (Subsection~\ref{subsection:ordinal}).
We shall provide examples of compact spaces~$K$ which admit local compensation but 
no $B_{\dualC}$-compensation function. Those examples and Proposition~\ref{pro:bola} below
make clear that there exist compact spaces which do not admit local compensation. 

\begin{pro}\label{pro:bola}
Let $K$ be a compact space. If $B_{\dualC}$ (equipped with the $\weak^*$-topology)
admits local compensation, then $K$ admits a $B_{\dualC}$-compensation function.
\end{pro}
\begin{proof} Write $L:=B_{\dualC}$ and let $\phi\colon K \to L$ be defined by $\phi(t):=\delta_t$, so that
$\phi$ is a homeomorphism onto~$\phi(K)$. Let $F\colon L \to B_{\dualC[L]}$ be the function defined by
$$
	\la F(\mu),f \ra=
	\la \mu, f \circ \phi \ra
	\quad
	\mbox{for every }f\in C(L) \mbox{ and }\mu \in L.
$$
Observe that $F(\mu)(D)=\mu(\phi^{-1}(D))$ for every $\mu\in L$ and every Borel set $D \subseteq L$.	
Since $F$ is $\weak^*$-continuous and $L$ admits local compensation, there is a $\weak^*$-continuous
function $G\colon L \to B_{\dualC[L]}$ such that $G(\mu)$ is a compensation
of~$F(\mu)$ for every $\mu\in L$. Let $S\colon \dualC[L] \to \dualC[\phi(K)]$ be the restriction operator
and $U\colon C(K) \to C(\phi(K))$ the isometric isomorphism given by $U(g):=g\circ \phi^{-1}$.
Define $\xi\colon L \to L$ by $\xi:=U^*\circ S\circ G$. 

We shall check that $\xi$ is a $B_{\dualC}$-compensation function.
Note first that $S\circ G$ is $\weak^*$-$\weak^*$-continuous, thanks to the $\weak^*$-continuity of~$G$
and the fact that 
\begin{equation}\label{inclusionsFG}
	\supp(G(\mu)) \subseteq \supp(F(\mu)) \subseteq \phi(K) \quad
	\mbox{for every }\mu\in L.
\end{equation}
Hence $\xi$ is $\weak^*$-$\weak^*$-continuous as well. On the other hand,
take any $\mu\in L$. Since $G(\mu)$ is a compensation of~$F(\mu)$ and the inclusions~\eqref{inclusionsFG} hold,
it follows at once that $S(G(\mu))$ is a compensation of~$S(F(\mu))$
Therefore, $\xi(\mu)$ is a compensation of~$U^*(S(F(\mu)))=\mu$.
\end{proof}

To go a bit further when studying the existence of 
compensation functions, we introduce the following definition.

\begin{defi}\label{defi:closeness}
Let $K$ be a compact space. A {\em closeness function} for~$K$ is a continuous function 
$$
	c:\{(x,y,z)\in K^3 \colon \, y\neq z\}\to [-1,1]
$$ 
such that:
\begin{enumerate}
\item $c(x,y,z) = -c(x,z,y)$ whenever $y\neq z$;
\item $c(x,x,z)=1$ whenever $x\neq z$.
\end{enumerate}
\end{defi}

\begin{rem}
If $(K,\rho)$ is a compact metric space, then the formula 
$$
	c(x,y,z) := \frac{\rho(x,z)-\rho(x,y)}{\max\{\rho(x,y),\rho(x,z)\}} 
$$
provides a closeness function for~$K$.
\end{rem}

Next lemma gives a connection between closeness and compensation functions:

\begin{lem}\label{pro:CompensationImpliesCloseness}
Let $K$ be a compact space. If $K$ admits a $B_{\dualC}$-compensation function, 
then $K$ admits a closeness function.
\end{lem}

\begin{proof}
Fix a $B_{\dualC}$-compensation function $\xi: B_{\dualC} \to B_{\dualC}$. Define  
$$
	c\colon\{(x,y,z)\in K^3 : \, y\neq z\}\to \R, \quad
	c(x,y,z) := 1 - 6\cdot\xi(f(x,y,z))(\{y\}),
$$
where $f(x,y,z):=\frac{1}{3}(\delta_y+\delta_z - \delta_x)$.
We will check that $c$ is a closeness function for~$K$.

Fix $(x,y,z)\in K^3$ with $y\neq z$. Then $f(x,y,z)(K)=\frac{1}{3}>0$, hence
$$
	\xi(f(x,y,z))(K)=\frac{1}{3} \quad
	\mbox{and}\quad 0\leq \xi(f(x,y,z)) \leq (f(x,y,z))^+=\frac{1}{3}(\delta_y+\delta_z).
$$ 
In particular, $c(x,y,z)\in [-1,1]$. 
On one hand, if $x=y$ then
$$
	c(x,x,z) =  1 - 6\cdot \xi(f(x,x,z))(\{x\}) =
	1 - 2\cdot \delta_z(\{x\})=1-0 = 1.
$$ 
On the other hand, since ${\rm supp}(f(x,y,z)) \subseteq \{y,z\}$, we have 
$$
	\frac{1}{3}= \xi(f(x,y,z))(K) = \xi(f(x,y,z))(\{y\})
	+\xi(f(x,y,z))(\{z\}),
$$ 
therefore
$$
	c(x,y,z) = 6\cdot\xi(f(x,y,z))(\{z\})-1 = -c(x,z,y).
$$
We finally check that $c$ is continuous at $(x,y,z)$. Since $y\neq z$, there exist disjoint
open sets $V,W \subseteq K$ with $y\in V$, $z\in W$, and a continuous 
function $\phi\colon K\to [0,1]$ such that $\phi|_{V}\equiv 1$ and $\phi|_W\equiv 0$. 
Then for every $(x',y',z') \in K\times V\times W$ we have
\begin{equation}\label{cambiazo}
	\xi(f(x',y',z'))(\{y'\}) =
	\la \xi(f(x',y',z')),\phi \ra,
\end{equation}
because ${\rm supp}(\xi(f(x',y',z')) \subseteq \{y',z'\}$.
Equality \eqref{cambiazo} and the $\weak^*$-$\weak^*$-continuity of~$\xi$ 
imply that $c$ coincides with a continuous function on $K\times V\times W$, which is
an open neighborhood of $(x,y,z)$.
This shows that $c$ is a closeness function for~$K$.
\end{proof}

Part (i) of the following proposition was pointed out to us
by O.~Kalenda and is included here with his kind permission.

\begin{pro}\label{separable}
Let $K$ be a compact space admitting a closeness function. 
\begin{enumerate}
\item $K$ is first countable.
\item If $K$ is separable, then it is metrizable. 
\end{enumerate}
\end{pro}

\begin{proof} Let $c$ be a closeness function for~$K$. We begin by proving the following:

{\em CLAIM. If $x_0\in K$ belongs to the closure of a countable set $D\subseteq K\setminus \{x_0\}$, 
then $x_0$ is a~$\mathcal{G}_\delta$-point.}

Indeed, for every $z\in D$ we have $c(x_0,x_0,z) = 1$, hence we can take an open neighborhood $V_z$ of~$x_0$ 
such that $z\not \in V_z$ and $c(x_0,x,z)>0$ for all $x\in V_z$. We claim that $\bigcap_{z\in D} V_{z} = \{x_0\}$. By contradiction, 
suppose there is $x\in \bigcap_{z\in D} V_{z} \setminus \{x_0\}$. Then $c(x_0,x,z)>0$ for all $z\in D$. 
Since $x_0\in \overline{D}$ and $c$ is continuous, we get $c(x_0,x,x_0)\geq 0$, which contradicts that $c(x_0,x,x_0)=-1$.
This proves the claim.

(i) Our proof is by contradiction. Suppose there is $x\in K$ which is not a~$\mathcal{G}_\delta$-point.
Then we construct a sequence $(x_n)$ in~$K \setminus \{x\}$
and a decreasing sequence $(H_n)$ of closed $\mathcal{G}_\delta$-sets containing~$x$ as follows:
\begin{itemize}
\item Pick an arbitrary $x_1\in K \setminus \{x\}$.
\item Given $n\in \mathbb{N}$, set $H_n:=\bigcap_{j=1}^n\{y\in K: c(y,x,x_j)=1\}$. Then
$H_n$ is a closed $\mathcal{G}_\delta$-set containing~$x$.
Since $x$ is not a $\mathcal{G}_\delta$-point,  
we can take $x_{n+1}\in H_n \setminus \{x\}$.
\end{itemize}
Now let $\tilde{x}$ be a cluster point of~$(x_n)$. 
By the CLAIM above, $\tilde{x}\neq x$.
Since $\tilde{x}\in H_n$ for every $n\in \mathbb{N}$, we have $c(\tilde{x},x,x_n)=1$
for every $n\in \N$. From the continuity of~$c$ it follows
that $c(\tilde{x},x,\tilde{x})=1$, which contradicts
that $c(\tilde{x},x,\tilde{x})=-1$.

(ii) It is enough to find a countable subset of $C(K)$ that separates the points of~$K$. 
Let $C$ be a countable dense subset of $K$. For every $t,s\in C$ with $t\neq s$, let 
$f_{t,s} \in C(K)$ be defined by $f_{t,s}(x) := c(x,t,s)$. 
Let us check that the countable family 
$$
	\{f_{t,s}\colon \, t,s\in C, \, t\neq s\}
$$ 
separates the points of~$K$. Fix $y\neq z$ in $K$. Since $c(y,y,z) = 1$, there are disjoint open 
sets $V_1,W_1 \subseteq K$ such that $y\in V_1$, $z\in W_1$ and 
$$
	c(x',y',z') > 0 \quad\mbox{for every }(x',y',z')\in V_1 \times V_1 \times W_1.
$$
On the other hand, since $c(z,y,z) = -1$, there are 
disjoint open sets $V_2,W_2 \subseteq K$ such that $y\in V_2$, $z\in W_2$ and 
$$
	c(x',y',z') < 0 \quad\mbox{for every }(x',y',z')\in W_2 \times V_2 \times W_2.
$$ 
Pick $t\in V_1\cap V_2\cap C$ and $s\in W_1\cap W_2\cap C$. Then $f_{t,s}(y) = c(y,t,s)>0$ while $f_{t,s}(z) = c(z,t,s) <0$.
The proof is over.
\end{proof}

By combining Lemma~\ref{pro:CompensationImpliesCloseness}, 
Proposition~\ref{separable}(ii) and Theorem~\ref{theo:CompactMetric} we get:

\begin{cor}\label{cor:almostCharacter}
Let $K$ be a compact space. The following statements are equivalent:
\begin{enumerate}
\item $K$ is separable and admits a~$B_{\dualC}$-compensation function; 
\item $K$ is metrizable.
\end{enumerate}
\end{cor}

\subsection{One-point compactifications of discrete sets}\label{subsection:compactification}

Throughout this subsection $\Gamma$ is a non-empty set and we denote by 
$K:=A(\Gamma)=\Gamma\cup\{\infty\}$ 
the one-point compactification of $\Gamma$ equipped with the discrete topology. 
Since $K$ is scattered, every element of $\dualC[K]$ is of the form
$\sum_{t\in K}a_t\delta_t$ for some $(a_t)_{t\in K}\in \ell^1(K)$, \cite[Theorem~14.24]{fab-ultimo}.
It is well-known that a bounded net $(\mu_\alpha)$ in~$\dualC[K]$ is $\weak^*$-convergent 
to~$\mu \in \dualC[K]$ if and only if $\mu_\alpha(K) \to \mu(K)$
and $\mu_\alpha(\{\gamma\})\to\mu(\{\gamma\})$ for all $\gamma\in\Gamma$.
Note that if $\Gamma$ is uncountable, then $\infty$ is not
a~$\mathcal{G}_\delta$-point of~$K$ and so
Proposition~\ref{separable}(i) yields:

\begin{cor}\label{theo:AGammaUncountable}
If $\Gamma$ is uncountable set, then 
$A(\Gamma)$ does not admit a closeness function.
Hence, it neither admits a $B_{\dualC[A(\Gamma)]}$-compensation function.
\end{cor}

However, we have the following:

\begin{theo}\label{theo:AGamma}
$A(\Gamma)$ admits local compensation and therefore $C(A(\Gamma))$ has the BPB property for numerical radius.
\end{theo}

\begin{proof} The second statement will follow from Theorem~\ref{lBPBP-C(K)}
once we prove the first one. Let $F\in W(K)$. 
If $F(\infty)(K)<0$, then there is a finite set $\Gamma_1 \subseteq \Gamma$ such that $F(t)(K)<0$ for all $t\in K \setminus \Gamma_1$
(because the function $F(\cdot)(K)\colon K\to \mathbb{R}$ is continuous). 
Fix an arbitrary compensation $\mu_t$ of~$F(t)$ for every $t\in \Gamma_1$ (apply Remark~\ref{rem:ExistenceSingle}).
Define $\xi_{F}\in W(K)$ by 
$$
	\xi_{F}(t):=
	\begin{cases}
		\mu_t & \text{if $t\in \Gamma_1$},\\
		0 & \text{if $t\in K \setminus \Gamma_1$.}
	\end{cases}
$$
Clearly, $\xi_{F}$ is a compensation of~$F$.

Suppose now that $F(\infty)(K)\geq 0$. Since the 
function $F(\cdot)(K)$ is continuous, there is a countable set $A\subseteq \Gamma$ such that
\begin{equation}\label{lcond-estability}
	F(t)(K)=F(\infty)(K) \quad\text{for every }t\in K\setminus A.
\end{equation}
For every $t\in K$ the set $A_t:=\supp(F(t))$ is countable. Write 
$$
	\Gamma_0:=A_\infty\cap\Gamma=A_\infty\setminus \{\infty\}.
$$ 
For each $s\in\Gamma_0$, the function $F(\cdot)(\{s\})\colon K \to \mathbb{R}$ is continuous
(because $\{s\}$ is a clopen subset of~$K$) and so there is a countable set $B_s\subseteq K$ such that
\begin{equation}\label{lcond-evaluate}
	F(t)(\{s\})=F(\infty)(\{s\}) \quad \text{for every }t\in K\setminus B_s.
\end{equation}
The set $B:=(\bigcup_{s\in\Gamma_0}B_s)\cup A\cup\{\infty\}$ is countable, hence so is $\bigcup_{t\in B} A_t$
and therefore
$$
	N:=\overline{\bigcup_{t\in B} A_t}
$$ 
is a compact metrizable (countable) subset of~$K$. Observe that for every $t\in B$ we have
$\supp(F(t))\subseteq N$. An appeal to Corollary~\ref{cor:LEMITA}
ensures the existence of a $\weak^*$-continuous function $G\colon B \to \dualC$ 
such that $G(t)$ is a compensation of~$F(t)$ for every $t\in B$.
Write $\xi_0:=G(\infty)$. Let us define  
\begin{equation}\label{eq:definicionC}
C:=\left\{t\in K\setminus B\colon F(t)(\Gamma \setminus A_\infty)\leq 0\right\}
\end{equation}
and the mapping $\xi_F\colon K\to\dualC$ by 
\begin{equation}\label{eq:definicionXi}
	\xi_F(t)(\{s\}):=
        \begin{cases}
        	G(t)(\{s\}) &\text{ if } t\in B\text{ and }s\in K,\\
			\xi_0(\{s\}) &\text{ if }t\in C\text{ and }s\in K,\\ 
			\xi_0(\{s\})&\text{ if }t\notin B\cup C\text{ and }s\in\Gamma_0,\\
            \frac{\xi_0(\{\infty\})}{F(t)^+(K \setminus \Gamma_0)}F(t)^+(\{s\})&\text{ if } t\notin B\cup C\text{ and }s\in K \setminus \Gamma_0.
        \end{cases}
\end{equation}
Observe that $K \setminus \Gamma_0 \supseteq \Gamma \setminus A_\infty$, hence
$F(t)^+(K \setminus \Gamma_0)\geq F(t)(\Gamma \setminus A_\infty)>0$ whenever $t\in K \setminus (B\cup C)$, so $\xi_F$ is well-defined. 
Let us show that $\xi_F$ is a compensation of~$F$.

\emph{STEP 1. $\xi_F(t)$ is a compensation of~$F(t)$ for every $t\in K$.} 

This is obvious for $t\in B$ by the choice of~$G$.
In particular, $\xi_0(K)=F(\infty)(K)\geq 0$ and $0\leq \xi_0 \leq F(\infty)^+$.
Let us analyze what happens for $t\in K\setminus B$. 

\emph{Case 1:} Assume $t\in C$. Then $t\notin B\supseteq A$, so
$$
	\xi_F(t)(K)\stackrel{\eqref{eq:definicionXi}}{=}\xi_0(K)=F(\infty)(K)\stackrel{\eqref{lcond-estability}}{=}F(t)(K).
$$
On the other hand, take any $s\in K$. Then 
$$
	0\leq \xi_F(t)(\{s\})\stackrel{\eqref{eq:definicionXi}}{=}\xi_0(\{s\})\leq F(\infty)^+(\{s\}).
$$
We now distinguish several cases.
\begin{itemize}
\item If $s\in\Gamma_0$, then \eqref{lcond-evaluate} implies that $F(\infty)^+(\{s\})=F(t)^+(\{s\})$.
\item If $s\notin A_\infty$, then $F(\infty)^+(\{s\})=0 \leq F(t)^+(\{s\})$. 
\item If $\infty \in A_\infty$, then
\begin{equation}
	\begin{split}
	F(\infty)(\{\infty\})&=F(\infty)(K)-\sum_{r\in \Gamma_0}F(\infty)(\{r\})
	\stackrel{\eqref{lcond-estability}\&\eqref{lcond-evaluate}}=F(t)(K)-\sum_{r\in\Gamma_0} F(t)(\{r\})=\\
	&=F(t)(\{\infty\})+F(t)(\Gamma\setminus A_\infty)\stackrel{\eqref{eq:definicionC}}{\leq} F(t)(\{\infty\}).
\end{split}
\end{equation}
\end{itemize}
It follows that
$$
	0\leq \xi_F(t)(\{s\})\leq F(\infty)^+(\{s\})\leq F(t)^+(\{s\}) \quad \text{for all } s\in K,
$$
hence $0\leq \xi_F(t)\leq F(\infty)^+\leq F(t)^+$. Therefore, $\xi_F(t)$ is a compensation of~$F(t)$.

\emph{Case 2:} Assume $t\in K\setminus (B\cup C)$. Then
\begin{equation*}
	\begin{split}
		\xi_F(t)(K)=
		&\sum_{s\in K}\xi_F(t)(\{s\})
		\stackrel{\eqref{eq:definicionXi}}{=}\sum_{s\in \Gamma_0}\xi_0(\{s\})+
		\frac{\xi_0(\{\infty\})}{F(t)^+(K \setminus \Gamma_0)}\sum_{s\in K \setminus \Gamma_0}F(t)^+(\{s\})\\
		=&\xi_0(\Gamma_0)+\xi_0(\{\infty\})=\xi_0(\Gamma_0\cup\{\infty\})=\xi_0(K)=F(\infty)(K)
		\stackrel{\eqref{lcond-estability}}{=}F(t)(K).
	\end{split}
\end{equation*}
To prove that $\xi_F(t)$ is a compensation of~$F(t)$
it remains to check that $0\leq\xi_F(t)\leq F(t)^+$, which is equivalent to saying that 
$0\leq\xi_F(t)(\{s\})\leq F(t)^+(\{s\})$ for all $s\in K$.
To this end, we distinguish two cases:
\begin{itemize} 
\item If $s\in\Gamma_0$, then 
\[
	0\leq \xi_F(t)(\{s\})\stackrel{\eqref{eq:definicionXi}}{=}\xi_0(\{s\})\leq F(\infty)^+(\{s\})
	=\stackrel{\eqref{lcond-evaluate}}{=}F(t)^+(\{s\}).
\]

\item If $s\in K \setminus \Gamma_0$, then  
\begin{equation}\label{eqn:Jose2}
	0\leq \xi_F(t)(\{s\})\stackrel{\eqref{eq:definicionXi}}{=}
	\frac{\xi_0(\{\infty\})}{F(t)^+(K \setminus \Gamma_0)}F(t)^+(\{s\})\leq
	F(t)^+(\{s\}),
\end{equation}
because 
\[
	F(\infty)(\{\infty\})=
	F(\infty)(K)-\sum_{r\in \Gamma_0}F(\infty)(\{r\})
	\stackrel{\eqref{lcond-estability}\&\eqref{lcond-evaluate}}{=}
	F(t)(K\setminus \Gamma_0)
\]
yields the inequalities $\xi_0(\{\infty\})\leq F(\infty)^+(\{\infty\})\leq F(t)^+(K\setminus \Gamma_0)$.
\end{itemize}

\emph{STEP 2. $\xi_F$ is $\weak^*$-continuous.}
It suffices to check that $\xi_F$ is $\weak^*$-continuous when restricted to each of the closed sets
$B$, $\overline{C}$ and $\overline{K \setminus (B\cup C)}$. We 
already know that the restriction $\xi_F|_B=G$ is $\weak^*$-continuous. 
On the other hand, note that $\xi_F(t)=\xi_F(\infty)=\xi_0$ for every $t\in \overline{C}$, hence 
$\xi_F|_{\overline{C}}$ is $\weak^*$-continuous. 

Finally, let us show that $\xi_F|_{\overline{K\setminus(B\cup C)}}$ is also $\weak^*$-continuous. 
To this end, it suffices to show that, if $(t_\alpha)$ is a net in~$K\setminus(B\cup C)$ converging to~$\infty$, 
then $\xi_F(t_\alpha) \to \xi_F(\infty)=\xi_0$ with respect to the $\weak^*$-topology
of~$\dualC$, which is equivalent to saying that
$$
	\xi_F(t_\alpha)(K) \to \xi_0(K)
	\quad\mbox{and}\quad
	\xi_F(t_\alpha)(\{s\})\to\xi_0(\{s\})\text{ for all }s\in\Gamma
$$
(note that $(\xi_F(t_\alpha))$ is bounded, because $\xi_F(t)$ is a compensation of
$F(t)\in B_{\dualC}$ for every $t\in K$).
We know that $\xi_F(t)(K)=F(t)(K)$ for all $t\in K\setminus (B\cup C)$ 
(see the proof of Step~1), hence $\xi_F(t_\alpha)(K)=F(t_\alpha)(K) \to F(\infty)(K)=\xi_0(K)$.

Fix any $s\in\Gamma$. If $s\in\Gamma_0$, then $\xi_F(t_\alpha)(\{s\})=\xi_0(\{s\})=\xi_F(\infty)(\{s\})$
for all~$\alpha$. If $s\in \Gamma \setminus \Gamma_0$, then by \eqref{eqn:Jose2} we have
$$	
	0\leq \xi_F(t_\alpha)(\{s\})\leq F(t_\alpha)^+(\{s\}) \quad
	\mbox{for all }\alpha.
$$
Since $F(t_\alpha)^+(\{s\}) \to F(\infty)^+(\{s\})=0$ (because $s\notin A_\infty$), it follows that
$$
	\xi_F(t_\alpha)(\{s\})\to 0=\xi_0(\{s\}).
$$
This proves that $\xi_F$ is $\weak^*$-continuous and so 
$\xi_F$ is a compensation of~$F$. 
\end{proof}

\subsection{Ordinal intervals}~\label{subsection:ordinal}
Throughout this subsection we work with the ordinal interval $K:=[0,\omega_1]$,
which becomes a $0$-dimensional scattered compact space when equipped with its order topology.
Since $K$ is scattered, every element of $\dualC[K]$ is of the form
$\sum_{\alpha\in K}a_\alpha\delta_\alpha$ for some $(a_\alpha)_{\alpha\in K}\in \ell^1(K)$, \cite[Theorem~14.24]{fab-ultimo}.
We shall also need the following well-known fact, see e.g. \cite[3.1.27]{eng-J}.

\begin{lem}\label{constantomega1}
If $h\colon [0,\omega_1)\to \mathbb{R}$ is a continuous function, then there is $\alpha<\omega_1$ such that $h$ is constant on $[\alpha,\omega_1)$.
\end{lem}

Since $\omega_1$ is not a $\mathcal{G}_\delta$-point of~$K$,
an appeal to Proposition~\ref{separable}(i) yields:

\begin{cor}\label{pro:ordinal}
$[0,\omega_1]$ does not admit a closeness function. Hence, it neither admits a 
$B_{\mathcal{M}([0,\omega_1])}$-compensation function.
\end{cor}

On the other hand, we have:

\begin{theo}\label{theo:ordinalAA}
$[0,\omega_1]$ admits local compensation and therefore $C([0,\omega_1])$ has the BPB property for numerical radius.
\end{theo}
\begin{proof} The second statement will follow from Theorem~\ref{lBPBP-C(K)}
once we prove the first one. Write $Clop(K)$ to denote the algebra of all clopen subsets of~$K$.
Let $F\in W(K)$. For every $\alpha\in K$, define
\begin{equation}\label{eqn:salpha}
	s(\alpha) := \sup\{\gamma<\omega_1\colon F(\alpha)(\{\gamma\}) \neq 0\} < \omega_1.
\end{equation} 

{\sc Claim 1.} There exist $b\in [-1,1]$ and $\gamma_0<\omega_1$ such that 
\begin{equation}\label{eqn:AfterGamma0}
	F(\gamma)(\{\omega_1\}) = b
	\quad\mbox{whenever }\gamma_0\leq \gamma <\omega_1.
\end{equation} 

Proof of Claim 1. By Lemma~\ref{constantomega1}, we only have to check that
the function $F(\cdot)(\{\omega_1\})$ is continuous on $[0,\omega_1)$. To this end, it is enough to prove the continuity 
on $[0,\gamma)$ for every $\gamma<\omega_1$. Let $\beta := \sup\{s(\alpha) \colon \alpha<\gamma\}<\omega_1$. Notice that for every $\alpha<\gamma$ we have
$$
	F(\alpha)([\beta+1,\omega_1])=\sum_{\beta < \gamma \leq \omega_1}F(\alpha)(\{\gamma\})\stackrel{\eqref{eqn:salpha}}{=}
	F(\alpha)(\{\omega_1\}). 
$$ 
Since $[\beta+1,\omega_1]\in Clop(K)$ and $F$ is $\weak^*$-continuous, the previous equality
ensures that the function $F(\cdot)(\{\omega_1\})$ is continuous on $[0,\gamma)$.
This finishes the proof of Claim~1.\qed

{\sc Claim 2.} For every $\alpha<\omega_1$ there is $\alpha < \beta(\alpha) < \omega_1$ such that $F(\gamma)(A) = F(\omega_1)(A)$ 
for every $\beta(\alpha) \leq \gamma < \omega_1$ and every $A\in Clop(K)$ such that 
$A\subseteq [0,\alpha]$ or $A=[0,\omega_1]$.

Proof of Claim 2. Note that the set 
$$
	\mathcal{A}_\alpha:=\{A\in Clop(K)\colon \, A\subseteq [0,\alpha]\}\cup\{[0,\omega_1]\}
$$ 
is countable (because $[0,\alpha]$ is a compact metric space and so it has countably many clopen subsets).
For every $A\in Clop(K)$ the function $F(\cdot)(A)\colon [0,\omega_1] \to \mathbb{R}$ is continuous, 
hence it is constant on $[\beta_A,\omega_1]$ for some $\alpha < \beta_A<\omega_1$ (apply Lemma~\ref{constantomega1}).
Now, the proof of Claim~2 finishes by taking $\beta(\alpha):=\sup\{\beta_A\colon A\in \mathcal{A}_\alpha\} < \omega_1$.\qed

{\sc Definition.}
We next define by transfinite induction a strictly increasing $\omega_1$-sequence of ordinals 
$\{\lambda_i \colon \, i<\omega_1\} \subseteq [0,\omega_1)$.  
For convenience, we consider 
$$
	\lambda_{0}:=\max\{s(\omega_1),\gamma_0,\beta(0)\} < \omega_1
$$ 
as the starting point of the induction. 
If $i<\omega_1$ is a limit ordinal, then we set 
$$
	\lambda_i := \sup\{\lambda_j \colon j<i\}.
$$ 
In the successor case, $\lambda_{i+1}$ is defined as 
$$
	\lambda_{i+1} := \sup\{\alpha_n \colon \, n\in \N\}= \sup\{\beta_n \colon \, n\in \N\}
$$ 
where $\lambda_i =: \alpha_0 < \beta_0 < \alpha_1 < \beta_1 < \ldots < \lambda_{i+1}$
are defined as 
\begin{equation}\label{eqn:defi-betan}
	\beta_n := \beta(\alpha_n) \quad \mbox{(given by Claim~2)}
\end{equation}
and
\begin{equation}\label{eqn:definalphan}
	\alpha_n := \max\big\{\beta_{n-1}+1,\sup\{s(\gamma) \colon \, \gamma\leq\beta_{n-1}\}\big\}.
\end{equation}

{\sc Definition.} We set
$$
	\mu:=F(\omega_1)-F(\omega_1)(\{\omega_1\})\delta_{\omega_1}\in \dualC,
$$
\begin{equation}\label{defmualpha}
	\mu_\alpha:= F(\alpha) - \mu - b\delta_{\omega_1}=F(\alpha)-F(\omega_1)+a\delta_{\omega_1}\in \dualC, \quad
	\alpha\in [0,\omega_1],
\end{equation}
where we write $a := F(\omega_1)(\{\omega_1\})-b$.

{\sc Claim 3.} If $i <\omega_1$ and $\alpha\in (\lambda_i,\lambda_{i+1})$, then 
$\mu_\alpha$ is concentrated on $[\lambda_i,\lambda_{i+1}]$ with $\mu_\alpha(K)=a$.

Proof of Claim 3. 
By \eqref{eqn:AfterGamma0} (bear in mind that $\alpha>\lambda_{0}\geq \gamma_0$) we have 
\begin{equation}\label{eqn:final}
	F(\alpha)(\{\omega_1\}) = b.
\end{equation}
Let $\lambda_i = \alpha_0 < \beta_0 < \alpha_1 < \beta_1 < \ldots < \lambda_{i+1}$ be the 
sequence of ordinals that defines $\lambda_{i+1} = \sup\{\alpha_n\colon n\in \N\}=\sup\{\beta_n\colon n\in \N\}$. 
Pick $n\in \N$ such that $\alpha\leq\beta_{n-1}$. By~\eqref{eqn:definalphan} we have $s(\alpha)\leq \alpha_n<\lambda_{i+1}$, and so 
\begin{equation}\label{eqn:cola}
	F(\alpha)(\{\gamma\})=0
	\quad\mbox{for every }\lambda_{i+1} < \gamma<\omega_1.
\end{equation} 
Note that we also have
\begin{equation}\label{eqn:cola2}
	F(\omega_1)(\{\gamma\})=0
	\quad\mbox{for every }\lambda_{i+1} < \gamma<\omega_1,
\end{equation} 
because $\lambda_{i+1}>\lambda_{0}\geq s(\omega_1)$. We next prove that
\begin{equation}\label{eqn:inicio}
	F(\alpha)(\{\gamma\}) = F(\omega_1)(\{\gamma\}) \quad \mbox{for every }\gamma < \lambda_{i}.
\end{equation}
To this end, it suffices to check the equality for every $\gamma<\lambda_{j+1}$ and every
ordinal $j<i$. Let $\lambda_j = \alpha'_0 < \beta'_0 < \alpha'_1 < \beta'_1 < \ldots < \lambda_{j+1}$ 
be the sequence of ordinals that defines $\lambda_{j+1}= \sup\{\alpha'_n\colon n\in \N\}=\sup\{\beta'_n\colon n\in \N\}$. Then
$\gamma < \alpha'_n$ for some $n\in \N$ and so 
we can write $\{\gamma\}=\bigcap_{k\in \N}A_k$ for some decreasing sequence $(A_k)$ in $Clop(K)$
with $A_k \subseteq [0,\alpha'_n]$ for every $k\in \N$. Indeed, this is obvious if $\gamma=0$, while
for $\gamma\neq 0$ we have $\{\gamma\}=\bigcap\{[\gamma_1+1,\gamma_2]\colon\gamma_1<\gamma\leq \gamma_2<\alpha'_n\}$.
By the choice of $\beta(\alpha'_n)$ (Claim~2) and 
$$
	\beta(\alpha'_n)\stackrel{\eqref{eqn:defi-betan}}{=}\beta'_n < \lambda_{j+1} \leq \lambda_i < \alpha,
$$ 
we have $F(\alpha)(A_k)=F(\omega_1)(A_k)$ for every $k\in \N$ and so
$$
	F(\alpha)(\{\gamma\})=\lim_{k\to \infty} F(\alpha)(A_k)=
	\lim_{k\to \infty}F(\omega_1)(A_k)=F(\omega_1)(\{\gamma\}).
$$
This proves~\eqref{eqn:inicio}. 
Since $\alpha>\lambda_{0}\geq\beta(0)$, we have $F(\alpha)(K) = F(\omega_1)(K)$ (Claim~2)
and therefore $\mu_\alpha(K)=a$ (by~\eqref{defmualpha}). 
Finally, from \eqref{eqn:final}, \eqref{eqn:cola} and~\eqref{eqn:inicio} we get
$$
	\mu_\alpha(\{\gamma\})=
	F(\alpha)(\{\gamma\})-F(\omega_1)(\{\gamma\})+a\delta_{\omega_1}(\{\gamma\})=0
	\quad
	\mbox{for every }\gamma \in K \setminus [\lambda_i,\lambda_{i+1}].
$$
The proof of Claim~3 is over.\qed

{\sc Claim 4.} For every $1\leq i<\omega_1$ we have $\mu_{\lambda_i}=a \delta_{\lambda_i}$ and so
\begin{equation}\label{eqn:FLambda}
	F(\lambda_i) = \mu + a \delta_{\lambda_i} + b \delta_{\omega_1}.
\end{equation}

Proof of Claim 4. We proceed by transfinite induction on~$i$. 
The limit ordinal case follows from the $\weak^*$-continuity of $F$.
Now, suppose \eqref{eqn:FLambda} holds for some $1\leq i<\omega_1$
and let us prove it for~$i+1$.
Consider again the chain $\lambda_i = \alpha_0 < \beta_0 < \alpha_1 < \beta_1 < \ldots < \lambda_{i+1}$ 
that defines $\lambda_{i+1}$ as its supremum. By the $\weak^*$-continuity of~$F$, 
the sequence $(F(\beta_n))$ is $w^*$-convergent to~$F(\lambda_{i+1})$, which (by~\eqref{defmualpha}) is equivalent
to saying that 
$$
	\lim_{n\to \infty}\mu_{\beta_n}(A)=\mu_{\lambda_{i+1}}(A)
	\quad\mbox{for every }A \in Clop(K).
$$
By Claim~3, each $\mu_{\beta_n}$ is concentrated on~$[\lambda_i,\lambda_{i+1}]$ with 
$\mu_{\beta_n}(K)=a$. In particular, we get $\mu_{\lambda_{i+1}}(K)=a$.
In order to prove that $\mu_{\lambda_{i+1}}=a\delta_{\lambda_{i+1}}$ it only remains to check
that $\mu_{\lambda_{i+1}}$ is concentrated on~$\{\lambda_{i+1}\}$.
Fix any $A\in Clop(K)$ with $A \subseteq [0,\lambda_{i+1})$. Since $A$ is compact, we have $A \subseteq [0,\alpha_n)$
for some $n\in\N$. Then $F(\beta_m)(A)=F(\omega_1)(A)$ for every $m\geq n$
(by Claim~2 and~\eqref{eqn:defi-betan}) and so $F(\lambda_{i+1})(A)=F(\omega_1)(A)$, hence $\mu_{\lambda_{i+1}}(A)=0$ (by~\eqref{defmualpha}).
As $A$ is an arbitrary clopen set contained in~$[0,\lambda_{i+1})$, we conclude
that $\mu_{\lambda_{i+1}}$ is concentrated on~$[\lambda_{i+1},\omega_1]$. On the other hand,
if we take any $A\in Clop(K)$ with $A \subseteq (\lambda_{i+1},\omega_1]$, then
$\mu_{\beta_n}(A)=0$ for all $n\in \N$ and so $\mu_{\lambda_{i+1}}(A)=0$.
It follows that $\mu_{\lambda_{i+1}}$ is concentrated on~$\{\lambda_{i+1}\}$, which
finishes the proof of Claim~4. \qed

{\sc Claim 5.} The restriction of~$F$ to $[0,\lambda_1]$ admits a compensation.

Proof of Claim~5. Write $L:=[0,\lambda_1]\cup\{\omega_1\}$. Notice first that 
\begin{equation}\label{eqn:soporte}
	\supp(F(\alpha)) \subseteq L
	\quad\mbox{for every }\alpha \in [0,\lambda_1].
\end{equation}
Indeed, this is immediate for $\alpha=\lambda_1$ (by~\eqref{eqn:FLambda}, bearing in mind
that $\lambda_1>\lambda_{0}\geq s(\omega_1)$). Let
$\lambda_0 = \alpha_0 < \beta_0 < \alpha_1 < \beta_1 < \ldots < \lambda_{1}$ 
be the chain that defines $\lambda_{1}$ as its supremum. If we take any $\alpha < \lambda_1$, then 
$\alpha\leq\beta_{n-1}$ for some $n\in \N$ and so $s(\alpha)\leq\alpha_n<\lambda_{1}$
(by~\eqref{eqn:definalphan}), hence $\supp(F(\alpha)) \subseteq L$. 
This proves~\eqref{eqn:soporte}. Since $L$ is compact metrizable, the claim now follows from
Corollary~\ref{cor:LEMITA}. \qed

{\sc Claim~6.} There exist $0\leq a' \leq \max\{a,0\}$, $0\leq b' \leq \max\{b,0\}$ and 
$\nu\in \mathcal{M}(K)$ with $0\leq \nu \leq \mu^+$ such that:
\begin{enumerate}
\item $\nu+a'\delta_{\lambda_i}+b'\delta_{\omega_1}$ is a compensation of~$F(\lambda_i)$ for every $1\leq i < \omega_1$;
\item $\nu+(a'+b')\delta_{\omega_1}$ is a compensation of~$F(\omega_1)$.
\end{enumerate}
Proof of Claim~6. Write $c:=a+b=F(\omega_1)(\{\omega_1\})$. Observe that
$$
	F(\omega_1)=\mu+c\delta_{\omega_1}
	\quad\mbox{and}\quad
	F(\lambda_i)\stackrel{\eqref{eqn:FLambda}}{=}\mu+a\delta_{\lambda_i}+b\delta_{\omega_1}
	\mbox{ for every }1\leq i <\omega_1,
$$
hence $F(\lambda_i)(K)=F(\omega_1)(K)$ (this equality also follows from Claim~2). 
Thus, the statement of Claim~6 holds trivially if $F(\omega_1)(K)\leq 0$.
We assume that $F(\omega_1)(K)>0$ and distinguish two cases.

Case 1. If $c\geq 0$, choose $a'',b''\in \mathbb{R}$ such that 
$a''\delta_{\lambda_1}+b''\delta_{\omega_1}$ is a compensation of $a\delta_{\lambda_1}+b\delta_{\omega_1}$. 
Then $0\leq a'' \leq \max\{a,0\}$, $0\leq b'' \leq \max\{b,0\}$ and 
$a''+b''=a+b=c$. Set $\nu_0:=\mu+a''\delta_{\lambda_1}+b''\delta_{\omega_1}$
and note that $\nu_0(K)=\mu(K)+c=F(\omega_1)(K)>0$. Let $\nu_1$ be a compensation of~$\nu_0$. Then
$\nu_1(K)=\nu_0(K)$ and
$$
	0\leq \nu_1 \leq (\mu+a''\delta_{\lambda_1}+b''\delta_{\omega_1})^+
	=\mu^+ + a''\delta_{\lambda_1}+b''\delta_{\omega_1},
$$
so we can write $\nu_1=\nu+a'\delta_{\lambda_1}+b'\delta_{\omega_1}$
for some $0\leq a'\leq a''$, $0\leq b' \leq b''$ and 
$\nu\in \dualC$ with $0\leq \nu \leq \mu^+$. It is clear that
$\nu+a'\delta_{\lambda_i}+b'\delta_{\omega_1}$ is a compensation of $F(\lambda_i)$ for every $1\leq i < \omega_1$
and that $\nu+(a'+b')\delta_{\omega_1}$ is a compensation of $F(\omega_1)$.

Case 2. If $c<0$, then let $\nu$ be a compensation of~$F(\omega_1)$, so that
$\nu(K)=F(\omega_1)(K)$ and $0\leq \nu \leq (\mu+c\delta_{\omega_1})^+=\mu^+$.
In particular, $\nu$ is a compensation of~$F(\lambda_i)$ for every $1\leq i <\omega_1$,
so we can take $a'=b'=0$ to conclude the proof of Claim~6.\qed

{\sc Claim~7.} For every $1\leq i <\omega_1$ there is a $\weak^*$-continuous function
$$
	\xi_i\colon [\lambda_i,\lambda_{i+1}]\to \dualC
$$ 
such that $\xi_i(\alpha)$ is a compensation of~$\mu_\alpha$ for all $\alpha\in [\lambda_i,\lambda_{i+1}]$.  
 
Proof of Claim~7. $[\lambda_i,\lambda_{i+1}]$ is compact metrizable. Since
$\supp(\mu_\alpha) \subseteq [\lambda_i,\lambda_{i+1}]$ for every $\alpha\in [\lambda_i,\lambda_{i+1}]$
(Claims~3 and~4) and the mapping $\alpha \mapsto \mu_\alpha$ is 
$\weak^*$-continuous (see \eqref{defmualpha}), the existence of~$\xi_i$ follows
from Corollary~\ref{cor:LEMITA}. \qed

Let $G\colon[0,\lambda_1] \to \dualC$ be a compensation of $F|_{[0,\lambda_1]}$ (Claim~5).
We now define $\xi_F\colon K \to \dualC$ by
$$
	\xi_F(\alpha):=
	\begin{cases}
	G(\alpha) & \text{if $\alpha\in [0,\lambda_1]$,}\\
	\nu+a'\delta_{\lambda_i}+b'\delta_{\omega_1} & \text{if $\alpha=\lambda_i$ and $2\leq i<\omega_1$,}\\
	\nu + \tilde{a}\xi_i(\alpha) + b'\delta_{\omega_1} & \text{if $\alpha\in(\lambda_i,\lambda_{i+1})$ and $1\leq i<\omega_1$,}\\
	\nu + (a'+b')\delta_{\omega_1} & \text{if $\alpha=\omega_1$,}
	\end{cases}
$$
where $\tilde{a}:=\frac{a'}{a}$ if $a>0$ and $\tilde{a}:=0$ if $a\leq 0$. 

We next check that $\xi_F(\alpha)$ is a compensation of~$F(\alpha)$ for every~$\alpha\in K$. This is clear
for $\alpha\in [0,\lambda_1]$ (by the choice of~$G$) and 
$\alpha\in \{\lambda_i\colon 2\leq i < \omega_1\}\cup\{\infty\}$ (by Claim~6). Take 
$\alpha\in(\lambda_i,\lambda_{i+1})$ for some $1\leq i<\omega_1$. 
Then $\mu_\alpha$ is concentrated on $[\lambda_i,\lambda_{i+1}]$ with $\mu_\alpha(K)=a$ (Claim~3).
Since $\xi_i(\alpha)$ is a compensation of~$\mu_\alpha$ (Claim~7), we have
$\xi_i(\alpha)=0$ whenever $a\leq 0$, while $\xi_i(\alpha)(K)=a$ and $0\leq \xi_i(\alpha)\leq \mu_\alpha^+$ whenever $a>0$.
In any case, we have $\tilde{a}\xi_i(\alpha)(K)=a'$. Note also that $F(\alpha)(K) = F(\omega_1)(K)$ (by Claim~2, since $\alpha>\lambda_{0}\geq\beta(0)$).
We now distinguish two cases.
\begin{itemize}
\item If $F(\omega_1)(K)>0$, then
$$
	\xi_F(\alpha)(K)=(\nu + \tilde{a}\xi_i(\alpha) + b'\delta_{\omega_1})(K)=
	\nu(K)+a'+b'= F(\omega_1)(K)=F(\alpha)(K)
$$
(bear in mind that $\nu+(a'+b')\delta_{\omega_1}$ is a compensation of~$F(\omega_1)$, see Claim~6).
Since $\nu \leq \mu^+$, $\tilde{a}\xi_i(\alpha) \leq \xi_i(\alpha)\leq \mu_\alpha^+$ and $b'\leq \max\{b,0\}$, we conclude
that 
$$
	0\leq \xi_F(\alpha) \leq 
	\mu^+ + \mu_{\alpha}^+ + \max\{b,0\}\delta_{\omega_1}
	\stackrel{\eqref{defmualpha}}{=} F(\alpha)^+.
$$
This shows that $\xi_F(\alpha)$ is a compensation of~$F(\alpha)$.

\item If $F(\omega_1)(K) \leq 0$, then $\nu=0$ and $a'=b'=0$ (Claim~6), hence $\xi_F(\alpha)=0$
is the compensation of~$F(\alpha)$. 
\end{itemize}

Finally, we check that $\xi_F$ is $\weak^*$-continuous.
Observe that the continuity of~$\xi_F$ on the open set $[0,\lambda_1]\cup \bigcup_{1\leq i < \omega_1}(\lambda_i,\lambda_{i+1})$
follows at once from the $\weak^*$-continuity of~$G$ and the $\xi_i$s. On the other hand, for any $1\leq i <\omega_1$, we have
$\mu_{\lambda_{i+1}}=a\delta_{\lambda_{i+1}}$ (Claim~4), hence 
$\xi_i(\lambda_{i+1})=\max\{a,0\}\delta_{\lambda_{i+1}}$ and so
$\tilde{a}\xi_i(\lambda_{i+1})=a'\delta_{\lambda_{i+1}}$. The last equality and
the $\weak^*$-continuity of~$\xi_i$ at~$\lambda_{i+1}$ ensure
that $\xi_F$ is $\weak^*$-continuous at~$\lambda_{i+1}$. To finish the proof we 
show that $\xi_F$ is continuous at~$\omega_1$. Fix any $A\in Clop(K)$. By Lemma~\ref{constantomega1}
applied to the restriction of $\xi_F(\cdot)(A)$ to~$[0,\omega_1)$, there exists some $\alpha_A<\omega_1$ such that
$\xi_F(\alpha)(A)=\xi_F(\alpha_A)(A)=:x_A$ for all $\alpha_A\leq \alpha < \omega_1$. 
Choose $2\leq i_A<\omega_1$ such that $\lambda_i \geq \alpha_A$ for every $i_A \leq i < \omega_1$. Then
$x_A=\xi_F(\lambda_i)(A)=(\nu+a'\delta_{\lambda_i}+b'\delta_{\omega_1})(A)$
for every $i_A \leq i < \omega_1$, and so $x_A=(\nu+a'\delta_{\omega_1}+b'\delta_{\omega_1})(A)=\xi_F(\omega_1)(A)$. 

The proof of the theorem is over.
\end{proof}

\begin{rem}
Using essentially the same arguments, one can prove by induction that the ordinal interval $[0,\aleph_n]$ admits local compensation for each 
$n\in\N$. It is not so clear to us what happens at~$\aleph_\omega$.
\end{rem}

\subsection{Some open problems}\label{section5}

Let $K$ be an arbitrary compact space.

\begin{enumerate}
\item[(a)] Does $C(K)$ have the BPB property for numerical radius?
\item[(b)] Does $K$ admit local compensation if $C(K)$ has the BPB property for numerical radius?
\item[(c)] Is $K$ metrizable if it admits a compensation function?
\item[(d)] Does $K$ admit a compensation function if it admits a $B_{\dualC}$-compensation function?
\item[(e)] Does $K$ admit a $B_{\dualC}$-compensation function if it admits a closeness function?
\end{enumerate}

\subsection*{Acknowledgement}
We wish to thank O.~Kalenda for his 
valuable comments on Section~\ref{section4} which led us to improve some results
contained in a previous version of this manuscript.

\def\cprime{$'$}
\providecommand{\bysame}{\leavevmode\hbox to3em{\hrulefill}\thinspace}
\providecommand{\MR}{\relax\ifhmode\unskip\space\fi MR }
\providecommand{\MRhref}[2]{%
  \href{http://www.ams.org/mathscinet-getitem?mr=#1}{#2}
}
\providecommand{\href}[2]{#2}

\bibliographystyle{amsplain}

\end{document}